\documentclass[a4paper,11pt]{article}
\usepackage{amsmath, amsfonts, amscd, amssymb, amsthm, enumerate, xypic}

\def\bfB{\mathbf{B}}
\def\bfC{\mathbf{C}}

\newcommand{\simt}{\underset{2}{\sim}}
\newcommand{\Mat}{\operatorname{M}}
\newcommand{\OS}{\operatorname{OS}}
\newcommand{\ROS}{\operatorname{ROS}}
\newcommand{\BM}{\operatorname{BM}}
\newcommand{\LRBM}{\operatorname{LRBM}}
\newcommand{\FRBM}{\operatorname{FRBM}}
\newcommand{\LLDR}{\operatorname{LLDR}}
\newcommand{\DR}{\operatorname{DR}}

\newcommand{\Mata}{\operatorname{A}}

\newcommand{\id}{\operatorname{id}}
\newcommand{\GL}{\operatorname{GL}}
\newcommand{\Ortho}{\operatorname{O}}
\newcommand{\GO}{\operatorname{GO}}
\newcommand{\Pf}{\operatorname{Pf}}
\newcommand{\Ker}{\operatorname{Ker}}

\newcommand{\Vect}{\operatorname{span}}
\newcommand{\im}{\operatorname{Im}}
\newcommand{\urk}{\operatorname{urk}}

\newcommand{\trk}{\operatorname{trk}}

\newcommand{\rk}{\operatorname{rk}}

\renewcommand{\setminus}{\smallsetminus}

\def\K{\mathbb{K}}

\def\N{\mathbb{N}}

\renewcommand{\L}{\mathbb{L}}

\def\calA{\mathcal{A}}
\def\calB{\mathcal{B}}
\def\calC{\mathcal{C}}

\def\calH{\mathcal{H}}

\def\calL{\mathcal{L}}
\def\calM{\mathcal{M}}

\def\calP{\mathcal{P}}

\def\calS{\mathcal{S}}
\def\calT{\mathcal{T}}

\def\calV{\mathcal{V}}
\def\calW{\mathcal{W}}

\def\lcro{\mathopen{[\![}}
\def\rcro{\mathclose{]\!]}}

\theoremstyle{definition}
\newtheorem{Def}{Definition}[section]
\newtheorem{Not}[Def]{Notation}

\theoremstyle{plain}
\newtheorem{theo}{Theorem}[section]
\newtheorem{prop}[theo]{Proposition}
\newtheorem{cor}[theo]{Corollary}
\newtheorem{lemme}[theo]{Lemma}

\theoremstyle{plain}
\newtheorem{conj}{Conjecture}[section]

\theoremstyle{remark}
\newtheorem{Rems}{Remarks}[section]
\newtheorem{Rem}[Rems]{Remark}

\title{Local linear dependence seen through duality II}
\author{Cl\'ement de Seguins Pazzis\footnote{Universit\'e de Versailles Saint-Quentin-en-Yvelines, Laboratoire de Math\'ematiques
de Versailles, 45 avenue des Etats-Unis, 78035 Versailles cedex, France}
\footnote{e-mail address: dsp.prof@gmail.com}}

\begin{document}

\thispagestyle{plain}

\maketitle

\begin{abstract}
A vector space $\calS$ of linear operators between finite-dimensional vector spaces $U$ and $V$ is called locally linearly dependent
(in abbreviate form: LLD) when every vector $x \in U$ is annihilated by a non-zero operator in $\calS$.
By a duality argument, one sees that studying LLD operator spaces amounts to studying vector spaces of matrices with rank
less than the number of columns, or, alternatively, vector spaces of non-injective operators.

In this article, this insight is used to obtain classification results for LLD spaces of small dimension or large essential range
(the essential range being the sum of all the ranges of the operators in $\calS$). We show that such classification theorems
can be obtained by translating into the context of LLD spaces Atkinson's classification of primitive spaces
of bounded rank matrices; we also obtain a new classification theorem for such spaces that covers a range of dimensions for the essential
range that is roughly twice as large as that in Atkinson's theorem. In particular, we obtain a classification of all $4$-dimensional
LLD operator spaces for fields with more than $3$ elements
(beforehand, such a classification was known only for algebraically closed fields and in the context of primitive spaces of matrices of bounded rank).

These results are applied to obtain improved upper bounds for the maximal rank in a minimal LLD operator space.
\end{abstract}



\tableofcontents

\section{Introduction}

Throughout the paper, all the vector spaces are assumed to be finite-dimensional.

\subsection{Local linear dependence}

Let $U$ and $V$ be finite-dimensional vector spaces over a field $\K$, and
$\calS$ be a linear subspace of the space $\calL(U,V)$ of all linear maps from $U$ to $V$.
We say that $\calS$ is \textbf{locally linearly dependent} (LLD)
when every vector $x \in U$ is annihilated by some non-zero operator $f \in \calS$.
Given a positive integer $c$, we say that $\calS$ is $c$-locally linearly dependent ($c$-LLD)
when, for every vector $x \in U$, the linear subspace $\{f \in \calS : f(x)=0\}$ has dimension at least $c$.

Alternatively, a family $(f_1,\dots,f_n)$ is called LLD when, for every
$x \in U$, the family $(f_1(x),\dots,f_n(x))$ is linearly dependent in $V$.
Obviously, this property is satisfied if and only if either $f_1,\dots,f_n$ are linearly dependent in $\calL(U,V)$
or $\Vect(f_1,\dots,f_n)$ is an LLD operator space.
Moreover, if some linear subspace $W$ of $V$ contains the image of each $f_i$ and $\dim W<n$,
then $f_1,\dots,f_n$ are obviously LLD.

The following example plays a central part in this article:
let $\varphi : U \times U \rightarrow V$ be an alternating bilinear map (with $U$ non-zero), and assume that $\varphi$
is fully-regular, that is $V=\Vect \{\varphi(x,y) \mid (x,y)\in U^2\}$ and $\varphi(x,-) \neq 0$ for all non-zero vector $x \in U$.
Then, the linear subspace $\calS_\varphi:=\{\varphi(x,-)\mid x \in U\}$ of $\calL(U,V)$ is LLD as,
for every non-zero vector $x \in U$, one has $\varphi(x,x)=0$ with $\varphi(x,-) \neq 0$.
We shall say that $\calS_\varphi$ is an operator space of the alternating kind.
An obvious example is the one of the standard pairing $\varphi : U \times U \rightarrow U \wedge U$.

Two operator spaces $\calS \subset \calL(U,V)$ and $\calS' \subset \calL(U',V')$ are called \textbf{equivalent}, and we write
$\calS \sim \calS'$, when there are two isomorphisms $F : U \overset{\simeq}{\rightarrow} U'$ and $G : V' \overset{\simeq}{\rightarrow} V$
such that $\calS=\{G \circ g \circ F \mid g \in \calS'\}$, in which case we have a uniquely defined isomorphism
$H : \calS \overset{\simeq}{\rightarrow} \calS'$ such that
$$\forall f \in \calS, \quad f=G \circ H(f) \circ F.$$
The corresponding notion for spaces of rectangular matrices is the standard equivalence relation, where
two matrix subspaces $\calM \subset \Mat_{m,n}(\K)$ and $\calM'\subset \Mat_{p,q}(\K)$ are equivalent
if and only if $m=p$, $n=q$ and there are non-singular matrices $P \in \GL_m(\K)$ and $Q \in \GL_n(\K)$ such that
$\calM=P\calM'Q$. Obviously, the operator spaces $\calS$ and $\calS'$ are equivalent if and only if
they are represented (in arbitrary bases of $U$, $V$, $U'$ and $V'$) by equivalent matrix spaces, or, alternatively,
if and only if there are choices of bases of $U$, $V$, $U'$ and $V'$ for which the same space of matrices
represents both $\calS$ and $\calS'$.

Note that if $\calS \sim \calS'$ and $\calS$ is $c$-LLD, then $\calS'$ is also $c$-LLD.
Moreover, if $\calS$ is minimal among the $c$-LLD subspaces of $\calL(U,V)$ and $\calS \sim \calS'$, then
$\calS'$ is minimal among the $c$-LLD subspaces of $\calL(U',V')$.
By the classification of minimal $c$-LLD operator spaces, we mean their determination up to equivalence.

\subsection{The duality argument}\label{dualityargumentsection}

Let $\calS$ be a $n$-dimensional linear subspace of $\calL(U,V)$.
The adjoint map of the natural embedding $\calS \hookrightarrow \calL(U,V)$
is
$$x \in U \longmapsto  \bigl[f \mapsto f(x)\bigr] \in \calL(\calS,V),$$
and we denote its image by
$$\widehat{\calS}:=\bigl\{f \mapsto f(x) \mid x \in U\bigr\} \subset \calL(\calS,V).$$
One sees that $\calS$ is LLD if and only if
$\widehat{\calS}$ is \textbf{defective}, i.e.\ no operator in $\widehat{\calS}$ is injective, that is to say
$$\forall \varphi \in \widehat{\calS}, \quad \rk \varphi <n.$$
Similarly, $\calS$ is $c$-LLD if and only if $\widehat{\calS}$ is $c$-\textbf{defective},
i.e.\ the kernel of every operator in $\widehat{\calS}$ has dimension at least $c$, i.e.\
$$\forall \varphi \in \widehat{\calS}, \quad \rk \varphi \leq n-c.$$
By choosing respective bases of $\calS$ and $V$,
one represents $\widehat{\calS}$ as a linear subspace $\mathcal{M}$ of $\Mat_{m,n}(\K)$
(the space of $m \times n$ matrices with entries in $\K$), where $m:=\dim V$.
For $\calS$ to be LLD (respectively, $c$-LLD), it is necessary and sufficient that $\rk M \leq n-1$
(respectively, $\rk M \leq n-c$) for all $M \in \calM$.
Thus, the study of LLD operator spaces is tightly connected to the one of linear subspaces of matrices with rank bounded above,
a theory which has been developed in the last sixty years  (see e.g.\ \cite{AtkinsonPrim,AtkLloydPrim,Flanders})
and for which we now have efficient computational tools and classification theorems.

\subsection{Reduced operator spaces}

When studying LLD operator spaces, only \emph{reduced} ones matter:
given a linear subspace $\calS$ of $\calL(U,V)$, the
\textbf{kernel} of $\calS$ is defined as $\underset{f \in \calS}{\bigcap} \Ker f$, i.e.\ the set of common zeros of the
operators in $\calS$;
the \textbf{essential range} of $\calS$ is defined as $\underset{f \in \calS}{\sum} \im f$;
we say that $\calS$ is \textbf{reduced} when its kernel is $\{0\}$ and its essential range is $V$.

Denote by $U_0$ the kernel of $\calS$ and by $V_0$ its essential range.
For every operator $f \in \calS$, the inclusions $U_0 \subset \Ker f$ and $\im f \subset V_0$ show that $f$ induces a
linear operator
$$\overline{f} : [x] \in U/U_0 \longmapsto f(x) \in V_0.$$
We see that
$$\overline{\calS}:=\Bigl\{\overline{f} \;\mid\; f \in \calS\Bigr\}$$
is a linear subspace of $\calL(U/U_0,V_0)$ and $f \mapsto \overline{f}$ is a rank-preserving isomorphism from $\calS$ to $\overline{\calS}$.
Moreover, $\overline{\calS}$ is reduced, and we call it the \textbf{reduced space attached to $\calS$.}
One checks that $\calS$ is $c$-LLD if and only if $\overline{\calS}$ is $c$-LLD.

Given a linear subspace $\calS'$ of $\calL(U',V')$
with kernel $U'_0$ and essential range $V'_0$,
one checks that
$$\calS \sim \calS' \; \Leftrightarrow \; \begin{cases}
\overline{\calS} \sim \overline{\calS'} \\
\dim U=\dim U' \\
\dim V=\dim V'.
\end{cases}$$
It ensues that the classification of minimal $c$-LLD operator spaces amounts to the one of
minimal reduced $c$-LLD operator spaces.

\subsection{Main goals}

In a previous work \cite{dSPLLD1}, we have shown how the above duality argument reduces the study of LLD operator spaces
to the one of spaces of non-injective operators. This was used to rediscover and expand the theory
of the minimal rank in LLD operator spaces and in non-reflexive operator spaces.
In this article, we use the duality argument to obtain significant advances in the following problems on LLD spaces:
\begin{enumerate}[(1)]
\item Classify minimal LLD operator spaces \emph{up to equivalence}. Prior to this article,
such classifications were known only for $2$-dimensional and $3$-dimensional LLD spaces,
and for $n$-dimensional minimal LLD spaces with an essential range of
dimension $\dbinom{n}{2}$ and, in each case, provided that the underlying field be of large enough cardinality \cite{ChebotarSemrl}.

\item Give an upper bound on the maximal rank in a minimal LLD operator space.
\end{enumerate}

Recently, an intriguing connection between LLD operator spaces and spaces of nilpotent matrices \cite{dSPAtkinsontoGerstenhaber}
has been uncovered: this gives a strong motivation for studying Problem (1), as it is expected that
classification results for minimal LLD spaces
can lead to a greater understanding of spaces of nilpotent matrices with a dimension close to the critical one.

Problem (1) will be our main point of focus: with the duality argument,
we shall show that studying it essentially amounts to studying so-called \emph{semi-primitive} operator spaces,
a notion that is closely connected to Atkinson and Lloyd's theory of primitive operator spaces \cite{AtkinsonPrim,AtkLloydPrim} later rediscovered
by Eisenbud and Harris in the more restrictive setting of algebraically closed fields \cite{EisenbudHarris}.
Most of our study revolves around obtaining classification results for semi-primitive operator spaces that
connect such spaces -- when their essential range is large enough -- to operator spaces of the alternating kind.
Our results on Problem (2) will, for the most part, be obtained as corollaries of our classification results for Problem (1).

The following theorem, which is already known (see \cite{ChebotarSemrl}), will be reproved in the course of the article
(this is Corollary \ref{rangecor} from Section \ref{essentialrangesection}):

\begin{theo}\label{majodimintro}
Let $\calS \subset \calL(U,V)$ be a minimal reduced LLD operator space with $\# \K \geq \dim \calS$. Then,
$\dim V \leq \dbinom{\dim \calS}{2}$.
\end{theo}

In \cite{ChebotarSemrl}, Chebotar and \v Semrl went on to classify the spaces for which $V$ has the critical dimension
$\dbinom{\dim \calS}{2}$. We shall obtain this as a very special case of the following two classification results:
the first one will be deduced from a generalized version of Atkinson's classification theorem of \cite{AtkinsonPrim}
(which we will reprove); the second one is entirely new.

\begin{theo}[First classification theorem for LLD spaces with large essential range]\label{classtheo1intro}
Let $\calS \subset \calL(U,V)$ be a minimal reduced LLD operator space with dimension $n$.
Assume that $\dim V >1+\dbinom{n-1}{2}$ and $\# \K \geq n \geq 3$. Then, $\calS$ is of the alternating kind.
\end{theo}

\begin{theo}[Second classification theorem for LLD spaces with large essential range]\label{classtheo2intro}
Let $\calS \subset \calL(U,V)$ be a minimal reduced LLD operator space with dimension $n$.
Assume that $\dim V >3+\dbinom{n-2}{2}$ and $\# \K \geq n \geq 3$. Then:
\begin{itemize}
\item Either $\calS$ is of the alternating kind;
\item Or there is a rank $1$ operator $g \in \calS$
such that, for the canonical projection $\pi : V \twoheadrightarrow V/\im g$, the operator space
$\{\pi \circ f \mid f \in \calS\}$ is an LLD subspace of $\calL(U, V/\im g)$ of the alternating kind
(with dimension $n-1$).
\end{itemize}
\end{theo}

Theorem \ref{classtheo1intro} is Corollary \ref{classcor1} from Section \ref{classtheo1Section}, while Theorem \ref{classtheo2intro}
is Corollary \ref{classcor2} from Section \ref{classtheo2Section}.

It is possible to obtain a classification theorem with the improved lower bound $6+\dbinom{n-3}{2}$,
but both its statement and its proof are way more technical than the ones of the above theorems:
this will be the topic of a separate article for which the present one will serve as a foundation.

\subsection{Additional definition and notation}

Throughout the article, the set of non-negative integers is denoted by $\N$ (following the French convention).

Given positive integers $m$ and $n$, one denotes by $\Mat_{m,n}(\K)$ the vector space of all matrices with
$m$ rows, $n$ columns and entries in $\K$. In particular, one sets $\Mat_n(\K):=\Mat_{n,n}(\K)$.
One defines $\Mata_n(\K)$ as the linear subspace of alternating matrices of $\Mat_n(\K)$.

\begin{Def}
Let $\calS$ be a linear subspace of $\calL(U,V)$.
The maximal rank of an operator in $\calS$ is called the \textbf{upper-rank} of $\calS$
and denoted by $\urk(\calS)$.

For $x \in \calS$, one defines the space $\calS x:=\{f(x) \mid f \in \calS\}$; one sets
$$\trk(\calS):=\max_{x \in U} \,\dim \calS x,$$
called the \textbf{transitive rank} of $\calS$. \\
For linear subspaces of matrices, we define the same notation by identifying $\Mat_{m,n}(\K)$
with $\calL(\K^n,\K^m)$ in the standard way.
\end{Def}

It is essential to note the following connection between transitive ranks and upper-ranks:
$$\trk \calS=\urk \widehat{\calS} \quad \text{and} \quad \urk \calS=\trk \widehat{\calS}.$$

\begin{Def}
Let $U_1,\dots,U_n,V$ be vector spaces, and
$f : U_1 \times \cdots \times U_n \rightarrow V$ be an $n$-linear map.
We say that $f$ is \textbf{regular with respect to the $i$-th argument}
when, for every non-zero vector $x_i \in U_i$, the map
$$f(-,\dots,-,x_i,-,\dots,-) : U_1 \times \cdots \times U_{i-1} \times U_{i+1}\times \cdots U_n
\longrightarrow V$$
is non-zero, that is $x_i \in U_i \mapsto f(-,\dots,-,x_i,-,\dots,-)$ is one-to-one. \\
In the special case when $n=2$, we say that $f$ is \textbf{left-regular} (respectively, \textbf{right-regular})
when it is regular with respect to the first argument (respectively, to the second argument). \\
We say that $f$ is \textbf{essentially surjective} when
$$\Vect \bigl\{f(x_1,\dots,x_n) \mid (x_1,\dots,x_n) \in U_1\times \cdots \times U_n\bigr\}=V.$$
We say that $f$ is \textbf{fully-regular} when it is essentially surjective and regular with respect to all its arguments.
\end{Def}

\begin{Rem}
Note that $f$ is essentially surjective if and only if the
$(n+1)$-linear form $\widetilde{f} : (x_1,\dots,x_n,\varphi) \in U_1\times \cdots \times U_n \times V^\star \longmapsto
\varphi(f(x_1,\dots,x_n)) \in \K$
induced by $f$ is regular with respect to the $(n+1)$-th argument. Therefore, $f$ is fully-regular if and only if
$\widetilde{f}$ is regular with respect to every single argument.
\end{Rem}

\vskip 3mm
At some point, we shall need to consider quadratic forms. Given scalars $a_1,\dots,a_n$ in $\K$,
the quadratic form $(x_1,\dots,x_n) \mapsto \underset{k=1}{\overset{n}{\sum}} a_k x_k^2$ on $\K^n$ is denoted by $\langle a_1,\dots,a_n\rangle$.
If $\K$ has characteristic $2$, then $[a,b]$ denotes the quadratic form $(x,y) \mapsto ax^2+xy+by^2$ on $\K^2$.
In any case, one sets
$$\K^{[2]}:=\bigl\{x^2 \mid x \in \K\bigr\} \quad \text{and} \quad  (\K^*)^{[2]}:=\bigl\{x^2 \mid x \in \K^*\bigr\}.$$

\subsection{Structure of the article}

In Section \ref{preliminarysection}, we shall remind the reader of some important key lemmas
that were reproved in \cite{dSPLLD1}.

The connection between primitive spaces of bounded rank matrices and LLD operator spaces
is fully established in Section \ref{LLDvsprimitive} with the help of categorical structures.
In that section, we shall also properly define semi-primitive matrix spaces, together with
the so-called \emph{column property} for matrix spaces with bounded rank. The column property is very important
as it is satisfied by semi-primitive spaces while being preserved by specific block extractions, thus allowing inductive proofs;
on the contrary, the semi-primitivity assumption does not behave well with such extractions.
We shall review some results of Atkinson and Lloyd on spaces with the column property,
which yields interesting consequences on LLD operator spaces, most notably Theorem \ref{majodimintro}.

Afterwards, we will tackle classification theorems for semi-primitive spaces: they will be obtained
as corollaries of classification theorems for spaces with the column property.
The first step consists of a systematic study of the operator space $\calS_\varphi$ associated with a fully-regular
alternating bilinear map $\varphi : U \times U \rightarrow V$. We shall rediscover
results from earlier works of Atkinson \cite{AtkinsonPrim} and Eisenbud and Harris \cite{EisenbudHarris}
with a more systematic approach. In particular, we will give sufficient conditions for such mappings to yield semi-primitive spaces
(or primitive spaces), and we will study the equivalence between such operator spaces and its relationship with the congruence
of alternating bilinear maps.  Interestingly, some of the results of this section are merely obtained as corollaries
to general results on semi-primitive operator spaces!

The next section is devoted to the statement and the proofs of the two main classification theorems
for spaces of matrices with the column property and to their applications to minimal LLD operator spaces with large essential range.
In Sections \ref{reductionlemmasection} and \ref{classtheo1Section}, we rediscover the basic tools for proving Atkinson's classification theorem and then
we generalize it to encompass spaces of matrices with the column property. This is applied
(Section \ref{classn=4section}) to provide a complete classification of all $4$-dimensional minimal LLD operator spaces,
provided that the underlying field has at least $4$ elements.
Before this article, such a classification was not known, and the corresponding classification of
primitive spaces of matrices with rank at most $3$ was known only for algebraically closed fields (see \cite{AtkinsonPrim} and \cite{EisenbudHarris}).
Section \ref{classtheo2Section} is devoted to a second order classification
theorem for minimal LLD operator spaces with large essential range: this theorem essentially doubles the range of dimensions
for which a classification is known and yields Theorem \ref{classtheo2intro} as a special case.

In the last section, we apply the previous classification theorems to obtain a significant improvement over
Meshulam and \v Semrl's estimate
for the maximal rank in a minimal LLD space \cite{MeshulamSemrlLAA}. With the above duality argument, this amounts
to estimate the transitive rank
of a semi-primitive space of bounded rank matrices; a key argument is that we have a very low estimate
for that transitive rank when the matrix space under consideration is of the alternating kind.

\section{Preliminary results from matrix theory}\label{preliminarysection}

Here, we review some useful technical lemmas on spaces of bounded rank matrices. They are all proved in Section 2 of \cite{dSPLLD1}.

\subsection{The generic rank lemma}

\begin{Def}
Let $\calS$ be a linear subspace of $\Mat_{m,n}(\K)$.
Given a basis $(A_1,\dots,A_s)$ of $\calS$ and independent indeterminates $\mathbf{x}_1,\dots,\mathbf{x}_s$,
the matrix $\mathbf{x}_1 A_1+\cdots+\mathbf{x}_s A_s$ of $\Mat_{m,n}\bigl(\K[\mathbf{x}_1,\dots,\mathbf{x}_s]\bigr)$
is called a \textbf{generic matrix} of $\calS$.

If we only assume that $A_1,\dots,A_s$ span $\calS$, then $\mathbf{x}_1 A_1+\cdots+\mathbf{x}_s A_s$
is called a \textbf{semi-generic matrix} of $\calS$.
\end{Def}

\noindent Note that the entries of $\mathbf{x}_1 A_1+\cdots+\mathbf{x}_s A_s$ are $1$-homogeneous polynomials
 in $\K[\mathbf{x}_1,\dots,\mathbf{x}_s]$.

\begin{lemme}[Generic rank lemma, Lemma 2.1 of \cite{dSPLLD1}]\label{genericrank}
Let $\calS$ be a linear subspace of $\Mat_{m,n}(\K)$ with $\# \K>\urk(\calS)$, and $\mathbf{A}$ be a semi-generic matrix
of $\calS$.
Then, $\urk(\calS)=\rk \mathbf{A}$.
\end{lemme}

\subsection{The Flanders-Atkinson lemma}

The following lemma was first discovered by Atkinson \cite{AtkinsonPrim}, extending a result of Flanders \cite{Flanders}
(see \cite{dSPLLD1} for a simplified proof).
It will be one of our main computational tools:

\begin{lemme}[Flanders-Atkinson lemma, Lemma 2.3 of \cite{dSPLLD1}]\label{Flanders}
Let $(m,n)\in (\N \setminus \{0\})^2$ and let $r \in \lcro 1,\min(m,n)\rcro$ be such that $\# \K > r$. \\
Set $J_r:=\begin{bmatrix}
I_r & [0]_{r \times (n-r)} \\
[0]_{(m-r) \times r} & [0]_{(m-r) \times (n-r)}
\end{bmatrix}$
and consider an arbitrary matrix
$M=\begin{bmatrix}
A & C \\
B & D
\end{bmatrix}$ of $\Mat_{m,n}(\K)$ with the same decomposition pattern. \\
Assume that $\urk(\Vect(J_r,M)) \leq r$.
Then,
$$D=0 \quad \text{and} \quad \forall k \in \N, \; BA^kC=0.$$
\end{lemme}

\subsection{The decomposition lemma}

The following decomposition principle for upper-ranks will be used frequently in the rest of the article:

\begin{lemme}[Lemma 2.4 of \cite{dSPLLD1}]\label{decompositionlemma}
Let $\calS$ be a linear subspace of $\Mat_{m,n}(\K)$ and assume that there exists a pair
$(r,s)$, with $1 \leq r\leq m$ and $1 \leq s\leq n$, such that
every matrix of $\calS$ splits up as
$$M=\begin{bmatrix}
[?]_{r \times s} & C(M) \\
B(M) & [0]_{(m-r) \times (n-s)}
\end{bmatrix}, \quad \text{with $B(M) \in \Mat_{m-r,s}(\K)$ and $C(M) \in \Mat_{r,n-s}(\K)$.}$$
Consider a semi-generic matrix of $\calS$ of the form
$$\mathbf{M}=\begin{bmatrix}
[?]_{r \times s} & \mathbf{C} \\
\mathbf{B} & [0]_{(m-r) \times (n-s)}
\end{bmatrix}.$$
Then, $\rk \mathbf{B}=\urk B(\calS)$, $\rk \mathbf{C}=\urk C(\calS)$ and
$$\urk(\calS) \geq \urk B(\calS) +\urk C(\calS) .$$
\end{lemme}

\section{Minimal LLD spaces versus primitive spaces of matrices of bounded rank}\label{LLDvsprimitive}

Here, we uncover a deep connection between minimal LLD operator spaces
and the primitive spaces of matrices of bounded rank studied by Atkinson and Lloyd \cite{AtkinsonPrim,AtkLloydPrim}.

We remind the reader that all the vector spaces that we consider are assumed to be finite-dimensional.

\subsection{The category of operator spaces}

We define the category of operator spaces, denoted by $\OS$, as follows:
\begin{itemize}
\item An object of $\OS$ is a triple $(U,V,\calS)$, where $U$ and $V$ are vector spaces (over $\K$)
and $\calS$ is a linear subspace of $\calL(U,V)$;
\item A morphism from $(U,V,\calS)$ to $(U',V',\calS')$ is a triple $(f,g,h)\in \calL(U',U) \times \calL(V,V') \times \calL(\calS',\calS)$
such that the following diagram is commutative for every $s' \in \calS'$:
$$\xymatrix{
U \ar[d]_{h(s')} & U' \ar[l]_{f}  \ar[d]^{s'} \\
V \ar[r]_g & V'.
}$$
The composition of morphisms is deduced from the one of linear maps.
\end{itemize}
Thus, two operator spaces $\calS \subset \calL(U,V)$ and $\calS' \subset \calL(U',V')$ are equivalent if and only if
the objects $(U,V,\calS)$ and $(U',V',\calS')$ of $\OS$ are isomorphic.

\subsection{The duality argument refined}

The category of bilinear maps, denoted by $\BM$, is defined as follows:
\begin{itemize}
\item An object of $\BM$ is a $4$-tuple $(U,V,W,B)$, where $U,V,W$ are vector spaces and $B : U \times V \rightarrow W$
is a bilinear map;
\item A morphism from $(U,V,W,B)$ to $(U',V',W',B')$
is a triple $(f,g,h) \in \calL(U',U) \times \calL(V',V) \times \calL(W,W')$ such that
$\forall (x,y) \in U' \times V', \; B'(x,y)=h\bigl(B(f(x),g(y))\bigr)$, i.e.\ the following diagram is commutative:
$$\xymatrix{
U \times V  \ar[d]_B & U' \times V' \ar[l]_{f\times g} \ar[d]^{B'} \\
W \ar[r]_h & W'.
}\quad $$
Again, the composition of morphisms is the obvious one.
\end{itemize}

Two bilinear maps $B : U \times V \rightarrow W$ and $B' : U' \times V' \rightarrow W'$ are called \textbf{equivalent}, and we write
$B \sim B'$, when the objects $(U,V,W,B)$ and $(U',V',W',B')$ are isomorphic in $\BM$.
This means that there are isomorphisms
$f : U' \overset{\simeq}{\rightarrow} U$, $g : V' \overset{\simeq}{\rightarrow} V$
and $h : W \overset{\simeq}{\rightarrow} W'$ such that
$\forall (x,y) \in U' \times V', \; B'(x,y)=h\bigl(B(f(x),g(y))\bigr)$.

\vskip 3mm
The transposition functor $T : \BM \rightarrow \BM$ assigns to every
object $(U,V,W,B)$ the object $(V,U,W,B')$, where $B' : (x,y) \mapsto B(y,x)$,
and assigns to every morphism $(f,g,h)$ the morphism $(g,f,h)$.

\begin{Not}
We denote by $\LRBM$ the full subcategory of $\BM$ in which the objects are the
$4$-tuples $(U,V,W,B)$ in which $B$ is left-regular.
\end{Not}

A linear subspace $\calS$ of $\calL(U,V)$ always gives rise to a bilinear mapping
$$B_\calS : (f,x) \in \calS \times U \longmapsto f(x) \in V,$$
and we note that:
\begin{enumerate}[(i)]
\item $B_\calS$ is always left-regular;
\item $B_\calS$ is right-regular if and only if the kernel of $\calS$ is $\{0\}$;
\item $B_\calS$ is essentially surjective if and only if the essential range of $\calS$ is $V$.
\end{enumerate}
In particular, $\calS$ is reduced if and only if $B_\calS$ is fully-regular.

We extend the above construction to a functor
$$B : \OS \longrightarrow \LRBM$$
by mapping the morphism $(U,V,\calS) \overset{(f,g,h)}{\longrightarrow} (U',V',\calS')$
to the morphism $(\calS,U,V,B_\calS) \overset{(h,f,g)}{\longrightarrow} (\calS',U',V',B_{\calS'})$.

\paragraph{}
Conversely, let $U$, $V$ and $W$ be vector spaces, and $B : U \times V \rightarrow W$ be a left-regular bilinear map.
Then, $x \in U \mapsto B(x,-) \in \calL(V,W)$ defines an isomorphism $\alpha_B$ from $U$ to a linear subspace $\calS_B$ of $\calL(V,W)$.
Given a morphism $(U,V,W,B) \overset{(f,g,h)}{\longrightarrow} (U',V',W',B')$ in the category $\LRBM$,
we define a morphism $\calS_{(f,g,h)} : (V,W,\calS_B) \rightarrow (V',W',\calS_{B'})$ in $\OS$ as $(g,h,\varphi)$, where the linear map
$\varphi : \calS_{B'} \rightarrow \calS_B$ is defined as $\alpha_B \circ f \circ (\alpha_{B'})^{-1}$.
One checks that this defines a functor
$$\calS : \LRBM \longrightarrow \OS.$$
Finally, one checks that $\calS \circ B$ is the identity functor on $\OS$,
while $B \circ \calS$ is naturally equivalent to the identity functor on $\LRBM$ through the natural equivalence
$(U,V,W,B) \mapsto (\alpha_B,\id_V,\id_W)$.
From here, we conclude:

\begin{prop}
Two operator spaces $\calS$ and $\calS'$ are equivalent if and only if $B_\calS \sim B_{\calS'}$.
Two left-regular bilinear maps $B$ and $B'$ are equivalent if and only if $\calS_B \sim \calS_{B'}$.
\end{prop}

Thus, classifying reduced operator spaces up to equivalence amounts to classifying fully-regular
bilinear maps up to equivalence.

\begin{Def}
We denote by $\ROS$ the full subcategory of $\OS$ in which the objects are the
triples $(U,V,\calS)$ such that $\calS$ is reduced. \\
We denote by $\FRBM$ the full subcategory of $\LRBM$ in which the objects
are the $4$-tuples $(U,V,W,B)$ such that $B$ is fully-regular.
\end{Def}

Restricting the functors $B$ and $\calS$ gives rise to functors $\ROS \rightarrow \FRBM$ and $\FRBM \rightarrow \ROS$
that are mutual inverses up to a natural equivalence.
The composite of the three functors $\ROS \rightarrow \FRBM$, $T : \FRBM \rightarrow \FRBM$ and
$\FRBM \rightarrow \ROS$ yields the \textbf{adjunction functor}:
$$\widehat{(-)} : \ROS \longrightarrow \ROS.$$
In short, for every object $(U,V,\calS)$ of $\ROS$, one has

$$\widehat{(U,V,\calS)}=(\calS,V,\widehat{\calS}), \quad \text{where
$\widehat{\calS}=\bigl\{f \mapsto f(x) \mid x \in U\bigr\}$,}$$
as in Section \ref{dualityargumentsection}. One checks that $\widehat{(-)}$ is an involution up to a natural equivalence, to the effect that
$\widehat{\widehat{\calS}} \sim \calS$ for every reduced operator space $\calS$;
moreover, given two reduced operator spaces $\calS$ and $\calS'$, one has
$$\calS \sim \calS' \;\Leftrightarrow\; \widehat{\calS} \sim \widehat{\calS'}.$$

We finish by interpreting the $c$-LLD condition for reduced operator spaces.

\begin{Def}
Let $B : U \times V \rightarrow W$ be a bilinear map.
We say that $B$ is \textbf{$c$-right-defective} (respectively, right-defective)
when the operator space $\bigl\{B(-,y)\mid y \in V\bigr\} \subset \calL(U,W)$ is $c$-defective (respectively, defective).
We define left-defectiveness likewise.
\end{Def}

A linear subspace $\calS$ of $\calL(U,V)$ is $c$-LLD if and only if the attached bilinear map
$B_\calS$ is $c$-right-defective. Conversely, given a $c$-right-defective and fully-regular bilinear map $B$,
the operator space $\calS_B$ is $c$-LLD, while $\calS_{T(B)}$ is reduced and $c$-defective.

Given a non-negative integer $c$, denote by $\LLDR(c)$ the full subcategory of $\ROS$ in which the objects are the $c$-LLD spaces,
and by $\DR(c)$ the full subcategory of $\ROS$ in which the objects are the $c$-defective spaces.
Then, one checks that the adjunction functor defines two functors, one from $\LLDR(c)$ to $\DR(c)$, and the other one
from $\DR(c)$ to $\LLDR(c)$, and that they are mutual inverses up to a natural equivalence.
From this, we draw a key principle:
\begin{center}
Classifying, up to equivalence, reduced $c$-LLD operator spaces between finite-dimensional
vector spaces amounts to classifying, up to equivalence, reduced $c$-defective
operator spaces between finite-dimensional vector spaces.
\end{center}

\subsection{Semi-primitive operator spaces}

Let $U$ and $V$ be finite-dimensional vector spaces, and $\calS$ be a linear subspace of $\calL(U,V)$.
Denote by $c$ the defectiveness index of $\calS$, so that $\calS$ is $c$-defective but not $(c+1)$-defective
(with the convention that every operator space is $0$-defective).

Consider the following conditions:
\begin{enumerate}[(i)]
\item There is no linear hyperplane $U'$ of $U$ such that $\{f_{|U'} \mid f \in \calS\}$ is a $c$-defective subspace of $\calL(U',V)$.
\item There is no surjective linear mapping $\pi : V \twoheadrightarrow V'$, with $\dim V'=\dim V-1$, such that $\{\pi \circ f\mid f \in \calS\}$
is a $(c+1)$-defective subspace of $\calL(U,V')$.
\end{enumerate}

\begin{Def}
We say that $\calS$ is \textbf{primitive} when it is reduced and satisfies conditions (i) and (ii) above. \\
We say that $\calS$ is \textbf{semi-primitive} when it is reduced and satisfies condition (i) above.
\end{Def}

Note that semi-primitivity and primitivity are both invariant under equivalence.

We define the notion of a semi-primitive/primitive $c$-defective subspace of $\Mat_{m,n}(\K)$ by referring
to the semi-primitivity/primitivity of the subspace of $\calL(\K^n,\K^m)$ associated with it in the canonical bases.
These conditions are invariant under the equivalence of matrix spaces, i.e.\ it is invariant by left and right-multiplication by
non-singular square matrices.
Another way to put things is to say that a subspace $\calS$ of $\Mat_{m,n}(\K)$ is primitive
if and only if it satisfies the following conditions:
\begin{itemize}
\item[(A)] $\calS$ is inequivalent to a space of matrices of the form
$\begin{bmatrix}
[?]_{m \times (n-1)} & [0]_{m \times 1}
\end{bmatrix}$;
\item[(B)] $\calS$ is inequivalent to a space of matrices of the form
$\begin{bmatrix}
[?]_{(m-1) \times n} \\
[0]_{1 \times n}
\end{bmatrix}$;
\item[(C)] $\calS$ is inequivalent to a space $\calS'$ of matrices of the form
$M=\begin{bmatrix}
H(M) & [?]_{m \times 1}
\end{bmatrix}$, where $H(\calS')$ is a subspace of $\Mat_{m,n-1}(\K)$ with $\urk H(\calS')<\urk \calS$;
\item[(D)] $\calS$ is inequivalent to a space $\calS'$ of matrices of the form
$M=\begin{bmatrix}
H(M) \\
[?]_{1 \times n}
\end{bmatrix}$, where $H(\calS')$ is a subspace of $\Mat_{m-1,n}(\K)$ with $\urk H(\calS')<\urk \calS$.
\end{itemize}
Moreover, $\calS$ is semi-primitive if and only if it satisfies conditions (A), (B) and (C);
$\calS$ is reduced if and only if it satisfies conditions (A) and (B) alone.
Remark also that $\calS$ is primitive if and only if both $\calS$ and $\calS^T$ are semi-primitive.

\begin{Rem}
Let $\calS$ be a semi-primitive subspace of $\Mat_{m,n}(\K)$, with $r:=\urk \calS$ and $n>0$.
Then, $\calS$ is defective; if $r=n$ indeed, then deleting the last column in the matrices of $\calS$ yields a space
in which every matrix has rank less than $r$. \\
It follows that if $\calS$ is primitive, then either $n=m=0$ or $r<\min(m,n)$.
\end{Rem}

\begin{Rem}
Let $\calV$ be a subspace of $\Mat_{m,n}(\K)$ with $\urk \calV=1$.
By a classical theorem of Schur, one of the spaces $\calV$ of $\calV^T$ is equivalent to a space
of matrices with all columns zero starting from the second one.
Therefore, if $\calV$ is reduced, then $m=1$ or $n=1$.
In particular, if we have a semi-primitive subspace $\calV$ of $\Mat_{m,n}(\K)$ with upper-rank $1$, then $m=1$ and $n\geq 2$.
\end{Rem}

Primitive matrix spaces with bounded rank were introduced by Atkinson and Lloyd \cite{AtkLloydPrim}
and rediscovered later by Eisenbud and Harris \cite{EisenbudHarris}.
They can be considered as the elementary pieces upon which
all the defective spaces are built, according to the first statement in Theorem 1 of \cite{AtkLloydPrim}.
The next result relates semi-primitive matrix spaces to primitive ones.

\begin{prop}\label{reductiontoprimitive}
Let $\calS$ be a semi-primitive linear subspace of $\Mat_{m,n}(\K)$ with upper-rank $r$.
Then, there are integers $p \in \lcro 0,r\rcro$ and $q \in \lcro 0,n\rcro$ such that
$q>r-p$, together with a space $\calS'\sim \calS$ of $\Mat_{m,n}(\K)$ in which every matrix splits up as
$$M=\begin{bmatrix}
[?]_{p \times q} & [?]_{p \times (n-q)} \\
H(M) & [0]_{(m-p) \times (n-q)}
\end{bmatrix},$$
and $H(\calS')$ is a primitive subspace of $\Mat_{m-p,q}(\K)$ with upper-rank $r-p$.
\end{prop}

\begin{proof}
The proof is essentially similar to that of the first statement of Theorem 1 of \cite{AtkLloydPrim}.
One takes the maximal integer $p$ for which $\calS$ is equivalent to a subspace $\calS'$ of $\Mat_{m,n}(\K)$
in which every matrix $M$ splits up as
$$M=\begin{bmatrix}
[?]_{p \times n} \\
Y(M)
\end{bmatrix}$$
and $\calT:=Y(\calS')\subset \Mat_{m-p,n}(\K)$ has upper-rank at most $r-p$. Obviously, $\urk \calT=r-p$.
Denoting by $t$ the dimension of the intersection of all the kernels of the matrices in $\calT$, we see that no generality is lost in
assuming that every matrix $N$ of $\calT$ splits up as
$$N=\begin{bmatrix}
A(N) & [0]_{(m-p) \times t}
\end{bmatrix} \quad \text{with $A(N) \in \Mat_{m-p,n-t}(\K)$,}$$
and no non-zero vector belongs to the kernel of every matrix of $A(\calT)$.
Obviously, $\urk A(\calT)=r-p$. Let us prove that $A(\calT)$ is primitive.

Obviously, $A(\calT)$ satisfies condition (A) in the definition of a primitive space.
As $\calS$ is reduced, one finds that the sum of all column spaces of the matrices in $\calT$ is $\K^{m-p}$, whence
$A(\calT)$ satisfies condition (B) in the definition of a primitive space.

Assume now that $A(\calT)$ does not satisfy condition (C). Then, without loss of generality, we
may assume that deleting the first column of every matrix of $A(\calT)$ yields a matrix whose
rank is less than $r-p$. Then, deleting the first column of every matrix of $\calS'$ yields a matrix
whose rank is less than $r$, contradicting the assumption that $\calS$ be semi-primitive.
Thus, $A(\calT)$ satisfies condition (C).

Assume finally that $A(\calT)$ does not satisfy condition (D). Then, no generality is lost in assuming that
deleting the first row of every matrix of $A(\calT)$ yields a matrix with rank less than $r-p$, in which case,
from every $M \in \calS'$, deleting the $p+1$ first rows yields a matrix whose rank is less than $r-p$,
contradicting the definition of $p$.

Therefore, $A(\calT)$ is primitive.
\end{proof}

We conclude this paragraph with an important statement that links semi-primitive $c$-defective operator
spaces to reduced minimal $c$-LLD operator spaces:

\begin{prop}
Let $\calS$ be a reduced LLD subspace of $\calL(U,V)$, and denote by $c$ the greatest integer for which $\calS$ is $c$-LLD.
Then, $\widehat{\calS}$ is semi-primitive if and only if $\calS$ is minimal among $c$-LLD subspaces of $\calL(U,V)$.
\end{prop}

\begin{proof}
Indeed, $\calS$ is minimal among $c$-LLD spaces if and only if no linear hyperplane $\calT$ of $\calS$ is $c$-LLD;
for any linear hyperplane $\calT$ of $\calS$, the condition that $\calT$ is not $c$-LLD
amounts to the property that, for some $x \in U$, the intersection of the kernel of
$\widehat{x} : f \mapsto f(x)$ with $\calT$ has dimension less than $c$,
i.e.\ $\bigl\{\widehat{x}_{|\calT} \mid x \in U\bigr\}$ is not $c$-defective.
As we know that $\widehat{\calS}$ is reduced, this condition means that $\widehat{\calS}$ is semi-primitive.
\end{proof}

\subsection{The column property for matrix spaces}

An important tool in the study of semi-primitive operator spaces is the so-called
\emph{column property}, which is weaker than semi-primitivity provided that the underlying field is large enough.
The transposed counterpart of this property was originally introduced by Atkinson and Lloyd \cite{AtkLloydPrim}, and
the notion itself is implicit in recent works of Chebotar, Meshulam and \v Semrl \cite{ChebotarSemrl,MeshulamSemrlLAA}.

\begin{Def}
Let $(r,s) \in \lcro 0,m\rcro \times \lcro 1,n\rcro$.
A linear subspace $\calS$ of $\Mat_{m,n}(\K)$ is called \textbf{$(r,s)$-decomposed} when every matrix of $\calS$
splits up as
$$M=\begin{bmatrix}
[?]_{r \times s} & C(M) \\
B(M) & [0]_{(m-r) \times (n-s)}
\end{bmatrix}, \; \text{where $B(M) \in \Mat_{m-r,s}(\K)$ and $C(M) \in \Mat_{r,n-s}(\K)$.}$$
The space $B(\calS)$ is then called the \textbf{lower space} of $\calS$ (as an $(r,s)$-decomposed space). \\
If $r>0$, the space $\calS$ is called $r$-\textbf{reduced} when $\urk \calS=r$ and $\calS$ contains
$$\begin{bmatrix}
I_r & [0]_{r \times (n-r)} \\
[0]_{(m-r) \times r} & [0]_{(m-r) \times (n-r)}
\end{bmatrix}.$$
\end{Def}

By the Flanders-Atkinson lemma, every $r$-reduced space is $(r,r)$-decomposed provided that $\# \K>r$.

\begin{Def}
A linear subspace $\calS$ of $\Mat_{m,n}(\K)$ has the \textbf{column property} when, for every
$(r,s) \in \lcro 0,m\rcro \times \lcro 1,n\rcro$ and every $(r,s)$-decomposed subspace $\calS' \sim \calS$,
the lower space of $\calS'$ is defective.
This property is obviously preserved by replacing $\calS$ with an equivalent space.

An operator space $\calS \subset \calL(U,V)$ has the column property when every
matrix space which represents it has the column property.
\end{Def}

In other words, the operator space $\calS$ has the column property if and only if, for every vector space $V'$ and every linear
surjection $\pi : V \twoheadrightarrow V'$, the reduced space associated with $\{\pi \circ f \mid f \in \calS\}$ is zero or defective. The special case $\pi=\id_V$ shows that an operator space with the column property is defective whenever its source space is non-zero.

\paragraph{}
The following two lemmas are essential to our study. In \cite{AtkLloydPrim}, they were obtained
in the course of larger proofs without being formally stated:

\begin{lemme}
Let $\calS$ be a semi-primitive subspace of $\Mat_{m,n}(\K)$, with
$\# \K >\urk(\calS)$ and $n>0$. Then, $\calS$ has the column property.
\end{lemme}

\begin{proof}
Assume, for some $(r,s) \in \lcro 0,m\rcro \times \lcro 1,n\rcro$, that some space $\calS'\sim \calS$ is $(r,s)$-decomposed; let us
write every matrix  $M \in \calS'$ as
$$M=\begin{bmatrix}
[?]_{r \times s} & C(M) \\
B(M) & [0]_{(m-r) \times (n-s)}
\end{bmatrix}.$$
Assume that $B(\calS')$ is non-defective, to the effect that $\urk B(\calS')=s$.
Then, Lemma \ref{decompositionlemma} yields $\urk C(\calS') \leq \urk \calS-s$.
This would contradict the assumption that $\calS$ be semi-primitive,
and more precisely that it should satisfy condition (C) (to see this, simply delete the first column).
\end{proof}

\begin{lemme}\label{inductioncolumnlemma}
Let $(r,s) \in \lcro 0,m\rcro \times \lcro 1,n\rcro$ and $\calS$ be an $(r,s)$-decomposed
subspace of $\Mat_{m,n}(\K)$ with the column property. Then, the lower space of $\calS$ has the column property.
\end{lemme}

\begin{proof}
The proof is given in the course of that of Lemma 6 of \cite{AtkinsonPrim}. We restate the argument as it is
both short and crucial to our study. \\
Let us write every matrix $M$ in $\calS$ as
$$M=\begin{bmatrix}
[?]_{r \times s} & [?]_{r \times (n-s)} \\
B(M) & [0]_{(m-r)\times (n-s)}
\end{bmatrix} \quad \text{with $B(M) \in \Mat_{m-r, s}(\K)$.}$$
Set $\calT:=B(\calS)$.
Let $(p,q) \in \lcro 0,m-r\rcro \times \lcro 1,s\rcro$ and $\calT' \sim \calT$ be a $(p,q)$-decomposed
linear subspace of $\Mat_{m-r,s}(\K)$. Without loss of generality, we may assume that $\calT'=\calT$, and then we
write every $N \in \calT$ as
$$N=\begin{bmatrix}
[?]_{p \times q} & [?]_{p \times (s-q)} \\
C(N) & [0]_{(m-r-p)\times (s-q)}
\end{bmatrix} \quad \text{with $C(N) \in \Mat_{m-r-p,q}(\K)$.}$$
Then, $\calS$ is $(r+p,q)$-decomposed with lower space $C(\calT)$, and hence
$C(\calT)$ is defective. Therefore, $\calT$ has the column property.
\end{proof}

This last lemma may fail if the column property is replaced with the semi-primitivity condition.
The fact that the column property is inherited by extracting lower spaces
is the prime motivation for introducing this property.
We finish by noting that the column property is really weaker than the semi-primitivity condition.
Consider indeed a reduced subspace $\calS \subset \Mat_{m,n}(\K)$ with the column property and $\urk \calS<m$, and
define $\calT$ as the space of all matrices of $\Mat_{m,n+1}(\K)$ of the form
$\begin{bmatrix}
S & Y
\end{bmatrix}$ with $S \in \calS$ and $Y \in \K^m$.
Then, $\calT$ is not semi-primitive (it fails to satisfy condition (C) in the definition of a semi-primitive space
as, obviously, $\urk \calT=\urk \calS+1$),
whereas $\calT$ is reduced and has the column property, which follows from the
observation that $\calT x=\K^m$ for all $x \in \K^{n+1} \setminus (\K^n \times \{0\})$ together with the assumption that
$\calS$ is reduced and has the column property.

\subsection{Two decomposition lemmas}

Here, we obtain two new decomposition principles for reduced spaces with the column property.
They will allow us to have a better grasp of such spaces in specific situations.

\begin{lemme}\label{1decomplemma}
Let $\calS$ be a reduced linear subspace of $\Mat_{m,n}(\K)$ with the column property and $m \geq 2$.
Set $r:=\urk(\calS)$ and assume that $\# \K>r$ and that there is a vector $x \in \K^n$ such that $\dim \calS x=1$.
Then, there is an integer $q \in \lcro r,n-1\rcro$ and a space
$\calS' \sim \calS$ such that every matrix $M$ of $\calS'$ splits up as
$$M=\begin{bmatrix}
[?]_{1 \times q} & [?]_{1 \times (n-q)} \\
H(M) & [0]_{(m-1) \times (n-q)}
\end{bmatrix},$$
and $H(\calS')$ is a reduced subspace of $\Mat_{m-1,q}(\K)$ with the column property and upper-rank $r-1$.
In particular, $\calS$ is non-primitive.
\end{lemme}

\begin{proof}
By assumption, the space $\calS$ is equivalent to a $(1,n-1)$-decomposed space.
Consider the minimal integer $q < n$ for which $\calS$ is equivalent to
a $(1,q)$-decomposed space; denote then by $\calS'$ such a subspace
and split every $M \in \calS'$ up as
$$M=\begin{bmatrix}
[?]_{1 \times q} & [?]_{1 \times (n-q)} \\
H(M) & [0]_{(m-1) \times (n-q)}
\end{bmatrix}, \quad \text{with $H(M) \in \Mat_{m-1,q}(\K)$.}$$
Since $\calS$ is reduced, the sum of all column spaces of the matrices of $H(\calS)$ is $\K^{m-1}$.
If some non-zero vector $y \in \K^q$ were annihilated by all the matrices $H(M)$, then
$\calS$ would be equivalent to a $(1,q-1)$-decomposed space, contradicting the minimality of $q$.
Thus, $H(\calS')$ is reduced. Moreover, Lemma \ref{inductioncolumnlemma} shows that $H(\calS')$ has the column property.
Finally, Lemma \ref{decompositionlemma} shows that $\urk H(\calS') \leq \urk(\calS)-1$ as $\calS$
is reduced, while obviously $\urk(\calS) \leq 1+\urk H(\calS')$.
Thus, $\urk H(\calS')=r-1$; since $\calS'$ has the column property, $H(\calS')$ is defective and hence $q \geq r$.

Deleting the first row of every matrix of $\calS'$ yields a matrix with rank at most $q-1$, whence $\calS$ is non-primitive.
\end{proof}

The next lemma can be useful in the special case of spaces that are $1$-defective but not $2$-defective:

\begin{lemme}[Thin decomposition lemma]\label{finedecomp}
Let $\calS$ be a reduced subspace of $\Mat_{m,n}(\K)$ with the column property,
set $r:=\urk(\calS)$ and assume that $r=n-1$ and $\# \K>r$.
Let $A \in \calS$ be with rank $r$. Let $x \in \Ker A \setminus \{0\}$, and set $p:=\dim \calS x$.
Then, there is space $\calT \sim \calS$, integers $s \in \lcro 0,m-r\rcro$, and
$t \in \lcro 0,r-p\rcro$, and linearly independent matrices $B_1,\dots,B_s$ of $\Mata_p(\K)$ such that:
\begin{enumerate}[(i)]
\item $\calT$ is $r$-reduced;
\item Every $M \in \calT$ splits up as
$$M=
\begin{bmatrix}
[?]_{p \times p} & [?]_{p \times t} & [?]_{p \times (r-p-t)} & C(M) \\
[?]_{(r-p) \times p} & [?]_{(r-p) \times t} & [?]_{(r-p) \times (r-p-t)} & [0]_{(r-p) \times 1} \\
R(M) & [?]_{s \times t} & [?]_{s \times (r-p-t)} & [0]_{s \times 1} \\
[0]_{(m-r-s) \times p} & H(M) & [0]_{(m-r-s) \times (r-p-t)} & [0]_{(m-r-s) \times 1}
\end{bmatrix};$$

\item For every $M$ in $\calT$, one has
$$R(M)=\begin{bmatrix}
C(M)^T B_1 \\
\vdots \\
C(M)^T B_s
\end{bmatrix};$$

\item If $t>0$, then the space $H(\calT) \subset \Mat_{m-r-s,t}(\K)$ is reduced, has the column property and
$\urk H(\calT) < t$.
\end{enumerate}
\end{lemme}

\begin{proof}
We lose no generality in assuming that $A=\begin{bmatrix}
I_r & [0]_{r \times 1} \\
[0]_{(m-r) \times r} & [0]_{(m-r) \times 1}
\end{bmatrix}$ and that $x$ is the last vector of the canonical basis of $\K^n$.
Thus, we may write every matrix $M \in \calS$ as
$$M=\begin{bmatrix}
[?]_{r \times r} & C'(M) \\
B'(M) & [0]_{(m-r) \times 1}
\end{bmatrix}, \quad \text{with $B'(M) \in \Mat_{m-r,r}(\K)$, $C'(M) \in \Mat_{r,1}(\K)$}$$
and $p=\dim C'(\calS)$.
Changing the bases once more, we see that no generality is lost in assuming that
$C'(\calS)=\K^p \times \{0\}$.
Then, we rewrite every $M \in \calS$ as
$$M=\begin{bmatrix}
[?]_{p \times p} & [?]_{p \times (r-p)} & C(M) \\
[?]_{(r-p) \times p} & [?]_{(r-p) \times (r-p)} & [0]_{(r-p) \times 1} \\
B(M) & [?]_{(m-r) \times (r-p)} & [0]_{(m-r) \times 1} \\
\end{bmatrix},$$
so that $C(\calS)=\K^p$.
The Flanders-Atkinson lemma yields:
\begin{equation}\label{nullproductidentity}
\forall M \in \calS, \; B(M)\,C(M)=0.
\end{equation}
Polarizing this quadratic identity yields
$$\forall (M,N) \in \calS^2, \; B(M)\,C(N)+B(N)\,C(M)=0.$$
For any $N \in \calS$ such that $C(N)=0$, we deduce that $B(N)C(\calS)=0$ and hence $B(N)=0$.
It follows that we have a linear map $\psi : \K^p \rightarrow \Mat_{m-r,p}(\K)$ such that
$\forall M \in \calS, \; B(M)=\psi(C(M))$, yielding matrices $B_1,\dots,B_{m-r}$ in $\Mat_p(\K)$ such that
$$\forall M \in \calS, \; B(M)=\begin{bmatrix}
C(M)^T B_1 \\
\vdots \\
C(M)^T B_{m-r}
\end{bmatrix}.$$
Since $C(\calS)=\K^p$, identity \eqref{nullproductidentity} shows that $B_1,\dots,B_{m-r}$ belong to $\Mata_p(\K)$.
Set $s:=\rk(B_1,\dots,B_{m-r})$.
Using row operations on the last $m-r$ rows, we see that no generality is lost in assuming that
$B_1,\dots,B_s$ are linearly independent and $B_{s+1}=\cdots=B_{m-r}=0$.
Then, for all $M \in \calS$, we have
$$B(M)=\begin{bmatrix}
R(M) \\
[0]_{(m-r-s) \times p}
\end{bmatrix}, \quad \text{where}\;
R(M)=\begin{bmatrix}
C(M)^T B_1 \\
\vdots \\
C(M)^T B_s
\end{bmatrix},$$
and we may write
$$M=\begin{bmatrix}
[?]_{p \times p} & [?]_{p \times (r-p)} & C(M) \\
[?]_{(r-p) \times p} & [?]_{(r-p) \times (r-p)} & [0]_{(r-p) \times 1} \\
R(M) & [?]_{s \times (r-p)} & [0]_{s \times 1} \\
[0]_{(m-r-s) \times p} & T(M) & [0]_{(m-r-s) \times 1} \\
\end{bmatrix},$$
where $T(M) \in \Mat_{m-r-s,r-p}(\K)$.

We finish by reducing $T(\calS)$. If $T(\calS)=\{0\}$, then we are done.
Assume now that $T(\calS)\neq \{0\}$.
Choosing a decomposition $\K^{r-p}=U \oplus \underset{M \in \calS}{\bigcap} \Ker T(M)$,
setting $t:=\dim U$ and changing the bases once more, we may assume that, for all $M \in \calS$, one has
$$T(M)=\begin{bmatrix}
H(M) & [0]_{(m-r-s) \times (r-p-t)}
\end{bmatrix}, \quad \text{with $H(M) \in \Mat_{m-r-s,t}(\K)$,}$$
and the intersection of the kernels of the matrices $H(M)$ is $\{0\}$.
Finally, as the sum of all column spaces of the matrices of $\calS$ is $\K^m$,
one obtains that the sum of all column spaces of the matrices of $H(\calS)$ is $\K^{m-r-s}$.
Therefore, $H(\calS)$ is reduced. Finally, Lemma \ref{inductioncolumnlemma} shows that $H(\calS)$ has the column property, and hence
$\urk H(\calS)<t$.
\end{proof}

\subsection{On the essential range of a minimal LLD operator space}\label{essentialrangesection}

Here, we recall and improve some known results on the essential range of a minimal LLD operator space.
Our starting point is the following lemma of Atkinson and Lloyd \cite[Lemma 6]{AtkLloydPrim}
(in \cite{AtkLloydPrim}, they actually state and prove the ``transposed" version):

\begin{theo}[Atkinson, Lloyd]\label{rangelemma}
Let $\calS$ be a reduced subspace of $\Mat_{m,n}(\K)$ with the column property.
Assume that $\# \K > r$.
Then, $m \leq \dbinom{r+1}{2}$. Moreover, $m \leq 1+\dbinom{r}{2}$ if $n>r+1$.
\end{theo}

The proof works by induction, using decompositions obtained by the Flanders-Atkinson lemma,
together with Lemma \ref{inductioncolumnlemma}: we shall not reproduce it.

For minimal LLD spaces, the duality argument applied to semi-primitive spaces yields the following corollary,
which was later rediscovered by Chebotar and \v Semrl \cite{ChebotarSemrl} (with a similar proof):

\begin{cor}\label{rangecor}
Let $\calS \subset \calL(U,V)$ be a reduced minimal LLD space.
Assume that $\# \K \geq \dim \calS$. Then, $\dim V \leq \dbinom{\dim \calS}{2}$.
\end{cor}

We finish by noting that in the situation considered by Chebotar and \v Semrl, an improved upper bound
on $\dim V$ may be found in some cases:

\begin{lemme}\label{finemaxdim}
Let $\calS$ be a reduced subspace of $\Mat_{m,n}(\K)$ with the column property,
set $r:=\urk(\calS)$ and assume that $r=n-1$ and $\# \K>r$.
Let $A \in \calS$ be a rank $r$ matrix. Let $x \in \Ker A \setminus \{0\}$, and set $p:=\dim \calS x$.
Then,
$$m\leq \binom{p+1}{2}+\binom{r-p+1}{2}.$$
\end{lemme}

\begin{proof}
Indeed, by Lemma \ref{finedecomp}, and using the same notation, we have
$s \leq \dim \Mata_p(\K)=\dbinom{p}{2}$. On the other hand, Theorem \ref{rangelemma} applied to $H(\calS')$ yields $m-r-s \leq
\dbinom{t}{2} \leq \dbinom{r-p}{2}$.
Thus,
$$m \leq r+\dbinom{p}{2}+\dbinom{r-p}{2}=\binom{p+1}{2}+\binom{r-p+1}{2}.$$
\end{proof}

\begin{Rem}\label{convexityremark}
With $r\geq 2$ fixed, the function
$$p \mapsto \dbinom{p+1}{2}+\dbinom{r-p+1}{2}$$
is polynomial of degree 2 with a positive coefficient along $p^2$, and
it is symmetric around $\frac{r}{2}$. It is therefore decreasing up to $\frac{r}{2}$, and increasing from that point on.
\end{Rem}

\section{Operator spaces of the alternating kind}\label{fromalternatingsection}

Let us restate that all the vector spaces are assumed finite-dimensional.

\subsection{The operator space attached to an alternating bilinear map}

Given a vector space $U$, we denote by
$\calB_2(U)$ the space of bilinear forms on $U$, and by $\calA_2(U)$ the space of alternating bilinear forms on $U$.

Let $\varphi : U \times U \rightarrow V$ be an essentially surjective alternating bilinear map.
Note that $\dim V \leq \dbinom{\dim U}{2}$. Note also that $\varphi$ is left-regular if and only if it is right-regular.

If $\varphi$ is left-regular, then $\calS_\varphi\sim \widehat{\calS_\varphi}$;
as $\varphi(x,-)$ vanishes at $x$ for all non-zero vector $x \in U$, one sees in this case that $\calS_\varphi$ is both an LLD subspace of $\calL(U,V)$
and a defective subspace of $\calL(U,V)$.
Otherwise, we may split $U=U_0 \oplus K$, where $K:=\{x \in U : \; \varphi(x,-)=0\}$,
and one is reduced to the fully-regular situation by
noting that $\varphi_{|U_0 \times U_0} : U_0 \times U_0 \rightarrow V$ is fully-regular.
It ensues that $\varphi$ is fully-regular whenever $\dim V>\dbinom{\dim U-1}{2}$.

\begin{Def}
An operator space is said to be \textbf{of the alternating kind} when it is equivalent to a space of the form
$\calS_\varphi$, where $\varphi : U \times U \rightarrow V$ is a fully-regular alternating bilinear map. \\
The spaces of matrices representing such operator spaces are also said to be of the alternating kind.
\end{Def}

A prime example is the one of the standard pairing
$$\varphi : (x,y) \in U \times U \longmapsto x \wedge y \in U \wedge U.$$
In that case, $\varphi$ is obviously fully-regular and $\calS_\varphi$ has upper-rank $\dim U-1$.

Let us recall the notion of congruent alternating bilinear maps:

\begin{Def}
Let $\varphi : U \times U \rightarrow V$ and $\psi : U' \times U' \rightarrow V'$ be alternating bilinear maps.
We say that $\varphi$ and $\psi$ are \textbf{congruent} when there are isomorphisms $f : U' \overset{\simeq}{\rightarrow} U$
and $h : V \overset{\simeq}{\rightarrow} V'$ such that
$$\forall (x,y) \in (U')^2, \; \psi(x,y)=h\bigl(\varphi(f(x),f(y))\bigr).$$
\end{Def}

Note that $\varphi$ and $\psi$ are equivalent if they are congruent, but the converse does not hold in general.

The rest of the section is laid out as follows: in Section \ref{describealternatingsection}, we describe
various ways to view essentially surjective alternating bilinear maps, and we discuss
the connection between them. In Section \ref{CSsemiprimitivesection}, we examine when the operator space associated with such a map is semi-primitive.
In particular, we will see that it is always so provided that the cardinality of the underlying field
is large enough.
In Section \ref{CStransitivesection}, we give a sufficient condition, based on the dimension of the target space,
for an operator space of the alternating kind to have the greatest possible upper-rank with respect to the dimension of the source space.
From that, we obtain a sufficient condition for semi-primitivity with no restriction on the cardinality of the underlying field.
In Section \ref{CSprimitivesection}, we use those results to obtain sufficient conditions for the primitivity of
an operator space of the alternating kind.
In Section \ref{equivalencetocongruencesection}, we give sufficient conditions for
the equivalence of the operator spaces $\calS_\varphi$ and $\calS_\psi$
to be equivalent to the congruence of the fully-regular alternating bilinear maps $\varphi$ and $\psi$. Finally, in Section \ref{trialitysection},
we study when an operator space of the alternating kind can be equivalent to the transpose of another such space.

\subsection{Alternative descriptions of essentially surjective alternating bilinear maps}\label{describealternatingsection}

For classification purposes, it is important to understand how
essentially surjective alternating bilinear maps may be obtained in practice.
First of all, the datum of such a map $\varphi : U \times U \rightarrow V$
is equivalent to that of a surjective linear map $\widetilde{\varphi} : U \wedge U \twoheadrightarrow V$
whose kernel we denote by $\calW_\varphi$. Then, $\varphi$ is congruent to
$$(x,y) \in U \times U \mapsto \bigl[x \wedge y\bigr] \in (U \wedge U)/\calW_\varphi.$$
By transposing, the datum of $\widetilde{\varphi}$ is seen to be equivalent to that of a linear injection
$\overline{\varphi} : V^\star \hookrightarrow (U \wedge U)^\star \simeq \calA_2(U)$,
the image of which we denote by $\calV_\varphi$. Transposing back, we see that $\varphi$ is congruent to
$$(x,y) \in U \times U \longmapsto \bigl[f \mapsto f(x,y)\bigr] \in (\calV_\varphi)^\star$$
Note that $\calW_\varphi$ corresponds to the dual orthogonal of $\calV_\varphi$ through the canonical isomorphism
between $(U \wedge U)^\star$ and $\calA_2(U)$.

The following results are obvious:

\begin{prop}
Let $\varphi : U \times U \rightarrow V$ be an essentially surjective alternating bilinear map.
The following conditions are equivalent:
\begin{enumerate}[(i)]
\item $\varphi$ is left-regular;
\item There is no vector $x \in U \setminus \{0\}$ such that $f(x,-)=0$ for all $f \in \calV_\varphi$;
\item There is no vector $x \in U \setminus \{0\}$ such that $\calW_\varphi$ contains $x \wedge U$.
\end{enumerate}
\end{prop}

\begin{prop}\label{significationcongruence}
Let $\varphi : U \times U \rightarrow V$ and $\psi : U' \times U' \rightarrow V'$ be essentially surjective alternating bilinear maps.
The following conditions are equivalent:
\begin{enumerate}[(i)]
\item $\varphi$ and $\psi$ are congruent;

\item There is a linear isomorphism $f : U \overset{\simeq}{\rightarrow} U'$ such that
$\calV_\varphi=\bigl\{(x,y) \mapsto B(f(x),f(y)) \mid B \in \calV_\psi\bigr\}$;

\item There is a linear isomorphism $f : U \overset{\simeq}{\rightarrow} U'$ such that
$\calW_\psi=(f \wedge f)(\calW_\varphi)$.
\end{enumerate}
\end{prop}

Setting $n:=\dim U$, $m:=\dim V$ and choosing a basis $\calB$ of $U$, the subspace $\calV_\varphi$
is represented in $\calB$ by a linear subspace $\calV_{\varphi,\calB}$ of $\Mata_n(\K)$ with dimension $m$.
When one varies the basis, the set of all such matrix spaces forms a congruence class of $m$-dimensional subspaces of $\Mata_n(\K)$.
Note that $\varphi$ is left-regular if and only if $\calV_{\varphi,\calB}$ is \textbf{incompressible}, i.e.\
it is not congruent to a subspace of $\Mata_{n-1}(\K) \oplus \{0\}$.

Similarly, we deduce an isomorphism $U \wedge U \overset{\simeq}{\longrightarrow} \Mata_n(\K)$
from the alternating bilinear map $(x,y) \mapsto XY^T-YX^T$, where $X$ and $Y$ are the respective matrices of $x$ and $y$ in $\calB$.
Thus, one recovers from $\calW_\varphi$ a linear subspace $\calW_{\psi,\calB}$ of $\Mata_n(\K)$ with dimension $\dbinom{n}{2}-m$.
Varying $\calB$, we find that a whole congruence class of $\Bigl(\dbinom{n}{2}-m\Bigr)$-dimensional subspaces of $\Mata_n(\K)$
is attached to $\varphi$.

We interpret Proposition \ref{significationcongruence} by saying that
$\varphi$ is determined up to congruence by the congruence class of $\calV_{\varphi,\calB}$,
and also by the one of $\calW_{\varphi,\calB}$.

\paragraph{}
Recall the standard non-degenerate symmetric bilinear form $\langle - \mid - \rangle$ on $\Mata_n(\K)$, defined, for
$A=(a_{i,j})$ and $B=(b_{i,j})$ in $\Mata_n(\K)$, by
$$\langle A\mid B\rangle:=\sum_{1 \leq i<j \leq n} a_{i,j} b_{i,j.}$$
One checks that, for all pairs $((x,y),b)$ in $U^2 \times \calA_2(U)$,
$$\langle XY^T-YX^T \mid M_\calB(b)\rangle=b(x,y),$$
where $X$ and $Y$ represent $x$ and $y$ in $\calB$.
Thus, $\calW_{\varphi,\calB}$ is the orthogonal of $\calV_{\varphi,\calB}$ for $\langle -\mid -\rangle$.

\begin{Rem}[An explicit matrix parametrization]\label{parametrage}
Let $\calH$ be a linear subspace of $\Mata_n(\K)$, equipped with a basis
$(A_1,\dots,A_m)$. The operator space associated with the essentially surjective alternating bilinear map
$$(X,Y)\in \K^n \times \K^n \longmapsto [M \mapsto X^T M Y]\in \calH^\star$$
is represented in the canonical basis of $\K^n$ and the dual basis of $(A_1,\dots,A_m)$ by the space of matrices
$$\Biggl\{\begin{bmatrix}
X^T A_1 \\
\vdots \\
X^T A_m
\end{bmatrix} \mid X \in \K^n\Biggr\}.$$
\end{Rem}

We close this paragraph by interpreting the dual operator space of the transpose of $\calS_\varphi$.

\begin{prop}\label{effetdunetransposition}
Let $\varphi : U \times U \rightarrow V$ be a fully-regular alternating bilinear map.
Then, through the natural isomorphism $f \in \calL(U,U^\star) \overset{\simeq}{\longrightarrow} [(x,y) \mapsto f(x)[y]] \in \calB_2(U)$,
the operator space $\widehat{(\calS_\varphi)^T}$ corresponds to $\calV_\varphi$ and
is therefore represented, in well-chosen bases, by a space of alternating matrices.
\end{prop}

\begin{proof}
The operator space $(\calS_\varphi)^T$ is the range of the linear map
$$\begin{cases}
U & \longrightarrow \calL(V^\star,U^\star) \\
x & \longmapsto \bigl[f \mapsto [y \mapsto f(\varphi(x,y))]\bigr].
\end{cases}$$
The dual operator space $\widehat{(\calS_\varphi)^T}$ is therefore the range of
$$\begin{cases}
V^\star & \longrightarrow \calL(U,U^\star) \\
f & \longmapsto \bigl[x \mapsto [y \mapsto f(\varphi(x,y))]\bigr],
\end{cases}$$
which, by identifying $\calL(U,U^\star)$ and $\calB_2(U)$ through the above canonical isomorphism, reads
$$\begin{cases}
V^\star & \longrightarrow \calB_2(U) \\
f & \longmapsto \bigl[(x,y) \mapsto (f\circ \varphi)(x,y)\bigr].
\end{cases}$$
As $\calV_\varphi$ is the range of this last map, the conclusion follows.
\end{proof}

\subsection{Sufficient conditions for semi-primitivity}\label{CSsemiprimitivesection}

First of all, we note that an operator space of the alternating kind has always the column property:

\begin{prop}\label{regularalternatingcolumn}
Let $\varphi : U \times U \rightarrow V$ be a fully-regular alternating bilinear map.
Then, $\calS_\varphi$ has the column property.
\end{prop}

\begin{proof}
Assume that there are bases $\calB=(e_1,\dots,e_n)$ and $\calC=(f_1,\dots,f_m)$, respectively, of
$U$ and $V$ such that the matrix space $\Mat_{\bfB,\bfC}(\calS_\varphi)$ is $(r,s)$-decomposed for some
$r \in \lcro 0,m\rcro$ and some $s \in \lcro 1,n\rcro$.
Set $U':=\Vect(e_1,\dots,e_s)$, $V':=\Vect(f_{r+1},\dots,f_m)$,
and denote by $\pi : V \twoheadrightarrow V'$ the projection alongside $\Vect(f_1,\dots,f_r)$.
The above assumptions mean that $\pi(\varphi(x,y))=0$ for all $x \in U$ and all $y \in \Vect(e_{s+1},\dots,e_n)$, and
we have to show that, for each $x \in U$, the operator $y \in U' \mapsto \pi(\varphi(x,y))$ has non-zero kernel.
However, given $x \in U \setminus \{0\}$, we may split $x=y+y'$ with $y \in U'$ and $y' \in \Vect(e_{s+1},\dots,e_n)$,
and we note that $\pi(\varphi(x,y))=-\pi(\varphi(x,y'))=0$, and hence either $y \neq 0$ and then we are done,
or $x \in \Vect(e_{s+1},\dots,e_n)$ and then $\pi(\varphi(x,e_1))=-\pi(\varphi(e_1,x))=0$. This finishes the proof.
\end{proof}

Now, we give a sufficient condition for $\calS_\varphi$ to be semi-primitive.

\begin{prop}\label{largefieldtosemiprimitive}
Let $\varphi : U \times U \rightarrow V$ be a fully-regular alternating bilinear map.
Set $r:=\urk \calS_\varphi$ and assume that $U$ is spanned by the vectors $x \in U$ for which $\rk \varphi(x,-)=r$.
Then, $\calS_\varphi$ is semi-primitive.
\end{prop}

\begin{proof}
Assume that $\calS_\varphi$ is not semi-primitive. As $\calS_\varphi$ is regular,
there must be a linear hyperplane $U_0$ of $U$ such that $\rk \varphi(x,-)_{|U_0}<r$ for all
$x \in U$. Let $x \in U \setminus U_0$. Then, $\dim \Ker \varphi(x,-)>\dim \Ker \varphi(x,-)_{| U_0}$ since $\varphi(x,x)=0$.
The rank theorem yields $\rk \varphi(x,-)_{|U_0} \geq \rk \varphi(x,-)$. Therefore,
$\rk \varphi(x,-)<r$ for every vector $x \in U \setminus U_0$, and hence all the
vectors $x \in U$ for which $\rk \varphi(x,-)=r$ belong to $U_0$, contradicting our assumptions.
\end{proof}

Now, we prove that the condition $\# \K> \urk \calS_\varphi$ is sufficient
for $\calS_\varphi$ to be semi-primitive:

\begin{lemme}\label{spanofrankr}
Let $\calS$ be a linear subspace of $\calL(U,V)$.
Set $r:=\urk \calS$ and assume that $\# \K>r$. Then, $\calS$ is spanned by its rank $r$ operators.
\end{lemme}

\begin{proof}
Let $\alpha$ be a linear form on $\calS$ such that $\alpha(f)=0$ for every rank $r$ operator $f \in \calS$.
Choose respective bases $\calB$ and $\calC$ of $U$ and $V$, and, for $f \in \calS$, consider the matrix
$M(f):=A(f) \oplus \alpha(f)$, where $A(f)=\Mat_{\bfB,\bfC}(f)$.
Then, $\rk M(f) \leq r$ for all $f \in \calS$. As $\#\K>r$, Lemma \ref{decompositionlemma} yields
$\urk \calS+\rk \alpha=\urk A(\calS)+\rk \alpha \leq r$.
Therefore, $\alpha=0$, which proves our claim.
\end{proof}

\begin{cor}\label{largefieldtosemiprimitive}
Let $\varphi : U \times U \rightarrow V$ be a fully-regular alternating bilinear map.
If $\# \K>\urk \calS_\varphi$, then $\calS_\varphi$ is semi-primitive.
\end{cor}

\subsection{A sufficient condition for transitivity}\label{CStransitivesection}

\begin{Def}
Let $\varphi : U \times U \rightarrow V$ be a fully-regular alternating bilinear map.
We say that $\varphi$ is \textbf{transitive} when some $x \in U$ satisfies
$\rk \varphi(x,-)=\dim U-1$, i.e.\ when $\urk \calS_\varphi=\dim U-1$, which is equivalent to having $\trk \calS_\varphi=\dim U-1$.
\end{Def}

In this paragraph, we discuss sufficient conditions on $\dim U$ and $\dim V$ for $\varphi$ to be transitive.
The first set of conditions works for all fields and gives a stronger conclusion than the mere transitivity.
In it, we need a simple definition:

\begin{Def}
A $2$-compound of a vector space $V$ is a subset of the form $H_1 \cup H_2$, where $H_1$ and $H_2$ are (linear)
hyperplanes of $V$.
\end{Def}

Note that a $2$-compound of $V$ is always a proper subset of $V$.

\begin{prop}\label{transitivityprop}
Let $\varphi : U \times U \rightarrow V$ be a fully-regular alternating bilinear map with
$\dim V>\dbinom{\dim U-1}{2}$. Let $C$ be either a $2$-compound of $U$ if $\# \K>2$, or a hyperplane of $U$ otherwise.
Then, $C$ does not contain all the vectors $x \in U$ for which $\rk \varphi(x,-)=\dim U-1$.
In particular, $\urk \calS_\varphi=\dim U-1$.
\end{prop}

In \cite{AtkLloydPrim}, Atkinson states a somewhat different version of Proposition \ref{transitivityprop},
with weaker assumptions and conclusions (the main assumption is that there is no proper linear subspace
$U_0$ of $U$ for which $V=\Phi(U_0 \times U_0)$, and the conclusion only involves hyperplanes).
Atkinson derives his result from a generalization of a theorem of Vaughan-Lee \cite[Section 2]{Vaughan-Lee}.
Here, we do not need the weaker assumption, however the
generalization to $2$-compounds will be very useful in what follows.
This justifies that we devote a few lines to an elementary self-contained proof of Proposition \ref{transitivityprop}.

Note that the considerations of Section \ref{describealternatingsection} show that the problem may be
entirely restated in terms of spaces of alternating matrices, so that Proposition \ref{transitivityprop} is
equivalent to the following statement:

\begin{lemme}\label{transitivitylemma}
Let $\calV$ be a linear subspace of $\Mata_n(\K)$.
Let $C$ be either a hyperplane of $\calV$ if $\# \K=2$, or a $2$-compound of $\calV$ otherwise.
Assume that the set $E:=\bigl\{x \in \K^n : \; \dim \calV x=n-1\bigr\}$ is included in $C$.
Then, $\dim \calV \leq \dbinom{n-1}{2}$.
\end{lemme}

\begin{proof}
The result is obvious if $n=1$. We proceed by induction: assume that $n \geq 2$ and that the result holds for $n-1$.
Note that $E \neq \K^n$, in any case.
Denoting by $(e_1,\dots,e_n)$ the canonical basis of $\K^n$, we see that no generality is lost in assuming that
$e_n \not\in C$, and that $C=\Vect(e_1,\dots,e_{n-1})$ if $\# \K=2$, and otherwise that
$C=\Vect(e_1,\dots,e_{n-1}) \cup H$ for some hyperplane $H$ of $\K^n$.
In this reduced situation, we write every matrix of $\calV$ as
$$M=\begin{bmatrix}
K(M) & -L(M)^T \\
L(M) & 0
\end{bmatrix}, \; \text{with $K(M) \in \Mata_{n-1}(\K)$ and $L(M) \in \Mat_{1,n-1}(\K)$,}$$
and set
$$\calW:=\{M \in \calV : \; L(M)=0\}.$$
The rank theorem yields
$$\dim \calV=\dim K(\calW)+\dim L(\calV).$$
If $L(\calV)=\{0\}$, then we readily have
$$\dim \calV =\dim K(\calW) \leq \dim \Mata_{n-1}(\K)=\binom{n-1}{2}.$$
From that point on, we shall identify $\K^{n-1}$ canonically
with the linear subspace $\K^{n-1} \times \{0\}$ of $\K^n$.

Assume now that $L(\calV) \neq \{0\}$.
Then, we may find a hyperplane $H_1$ of $\K^{n-1}$ which contains every $x \in \K^{n-1}$ satisfying
$x^T M e_n=0$ for all $M \in \calV$.

We contend that $H_1$ contains all the vectors $x \in \K^{n-1}$ for which $\dim K(\calW) x=n-2$.
Assume on the contrary that some
$x \in \K^{n-1}$ satisfies both $x^T \calV e_n \neq \{0\}$ and $\dim K(\calW) x=n-2$.
Let $t$ be a non-zero scalar. One sees that $K(\calW)x=\calW x =\calW (x+t\,e_n) \subset \calV (x+t\,e_n)$, while
$\dim K(\calW) x=n-2$ and $\dim (\calV (x+t\,e_n)) \leq n-1$. If $\dim \calV (x+t\,e_n)<n-1$, we would deduce that
$\calV (x+t\,e_n)=\calW x \subset \Mata_{n-1}(\K)x$, leading to
$x^T M e_n=t^{-1}\,x^TM(x+t\,e_n)=0$ for all $M \in \calV$.
Thus, $x+t\,e_n \in E$ for every non-zero scalar $t$. \\
On the other hand, if $\# \K=2$, then obviously $x+e_n \not\in C$ as $C=\K^{n-1}$;
if $\# \K>2$, then $C=\K^{n-1} \cup H$ for some linear hyperplane $H$ of $\K^n$ which does not contain $e_n$,
and hence there is at most one non-zero scalar $t$ such that $x+t\,e_n \in H$,
while $x+t\,e_n \not\in \K^{n-1}$ for all non-zero scalar $t$.
In any case, there is at least one non-zero scalar $t$ such that $x+t\,e_n \not\in C$, which contradicts the assumption that
$E \subset C$.

Thus, every vector $x \in \K^{n-1}$ for which $\dim K(\calW) x=n-2$ must belong to $H_1$.
By induction,
$$\dim K(\calW) \leq \binom{n-2}{2}.$$
Finally, as $e_n \not\in C$ and $E \subset C$, we have $\dim \calV e_n<n-1$, i.e.\ $\dim L(\calV)< n-1$.
We conclude that
$$\dim \calV \leq \binom{n-2}{2}+n-2=\binom{n-1}{2}.$$
\end{proof}

Lemma \ref{transitivitylemma} is optimal, which is demonstrated by the case of the subspace
$\Mata_{n-1}(\K) \oplus \{0\}$ of $\Mata_n(\K)$.

\vskip 3mm
Here is a corollary to Proposition \ref{transitivityprop} and Corollary \ref{largefieldtosemiprimitive}.

\begin{cor}\label{semiprimitivesmallfield}
Let $\varphi : U \times U \rightarrow V$ be a fully-regular alternating bilinear map with
$\dim V>\dbinom{\dim U-1}{2}$. Then, $\calS_\varphi$ is semi-primitive.
\end{cor}

For large enough fields, the condition on $\dim V$ in Proposition \ref{transitivityprop}
may be substantially lowered.
Interestingly, this is a consequence of Atkinson and Lloyd's lemma on the essential range of a space with the column property!

\begin{prop}\label{transitivityprop2}
Let $\varphi : U \times U \rightarrow V$ be a fully-regular alternating bilinear map. Set $n:=\dim U$ and assume that
$\dim V>1+\dbinom{n-2}{2}$ and $\# \K \geq n-1$.
Then, $\urk \calS_\varphi=n-1$.
\end{prop}

\begin{proof}
Set $r:=\urk \calS_\varphi$ and $m:=\dim V$ and assume that $r<n-1$. By Proposition \ref{regularalternatingcolumn},
the space $\calS_\varphi$ has the column property. Moreover, it is reduced since $\varphi$ is fully-regular.
As $r \leq n-2$, Theorem \ref{rangelemma} yields $m \leq 1+\dbinom{r}{2} \leq 1+\dbinom{n-2}{2}$,
which contradicts our assumptions.
\end{proof}

The following corollary, which is stated in terms of the transitive rank of a linear subspace of $\Mata_n(\K)$,
appears to be new:

\begin{cor}\label{transitivitylemma2}
Let $\calV$ be an incompressible linear subspace of $\Mata_n(\K)$ with dimension $m$.
Assume that $m>1+\dbinom{n-2}{2}$ and $\# \K \geq n-1$. Then, $\trk \calV=n-1$.
\end{cor}

That the lower bound $1+\dbinom{n-2}{2}$ is optimal is exemplified by the space $\Mata_2(\K) \oplus \Mata_{n-2}(\K)$.

\subsection{Sufficient conditions for primitivity}\label{CSprimitivesection}

Here, we apply the results of the previous sections to obtain necessary and sufficient conditions for the primitivity of
an operator space of the alternating kind.

First of all, the following lemma is straightforward:

\begin{lemme}
Let $\varphi : U \times U \rightarrow V$ be a fully-regular alternating bilinear map with $\dim V>0$.
Set $r:=\urk \calS_\varphi$.
Then, $\calS_\varphi$ is primitive if and only if $\calS_\varphi$ is semi-primitive and, for every
surjective linear map $\pi : V \twoheadrightarrow V'$ with $\dim V'=\dim V-1$, the operator space
$\calS_{\pi \circ \varphi}$ has upper-rank $r$.
\end{lemme}

As a consequence of Proposition \ref{transitivityprop} and Corollary \ref{semiprimitivesmallfield}, we deduce:

\begin{prop}\label{smallfieldprimitive}
Let $\varphi : U \times U \rightarrow V$ be a fully-regular alternating bilinear map
with $\dim V > 1+\dbinom{\dim U-1}{2}$. Then, $\calS_\varphi$ is primitive.
\end{prop}

For large fields and large target spaces, we can give improved necessary and sufficient conditions for $\calS_\varphi$ to be primitive:

\begin{prop}\label{CSprimitivite}
Let $\varphi : U \times U \rightarrow V$ be a fully-regular alternating bilinear map.
Assume that $\dim V>2+\dbinom{\dim U-2}{2}$ and $\# \K \geq \dim U-1$.
Then, $\calS_\varphi$ is non-primitive if and only if there exists a linear hyperplane $U_0 \subset U$
together with a non-zero vector $x \in U_0$ such that $\calW_\varphi$ contains $x \wedge U_0$.
\end{prop}

\begin{proof}
Set $n:=\dim U$ and $m:=\dim V$.
Without loss of generality, we may assume that $V=(U \wedge U)/\calW_\varphi$
and that $\varphi$ is the standard alternating bilinear map from $U \times U$ to $V$.
By Proposition \ref{transitivityprop}, we have $\urk \calS_\varphi=n-1$.

\begin{itemize}
\item Suppose that there exists a
linear hyperplane $U_0 \subset U$ together with a non-zero vector $x \in U_0$ such
that $\calW_\varphi$ contains $x \wedge U_0$.
Choose $y \in U \setminus U_0$.
Note that $\calW_\varphi$ cannot contain $x \wedge y$, for if it does then $\varphi(x,-)=0$, contradicting
the assumption that $\varphi$ be left-regular. \\
Thus $H:=\calW_\varphi \oplus \K(x \wedge y)$ has codimension $m-1$ in $U \wedge U$;
defining $\pi : (U \wedge U)/\calW_\varphi \twoheadrightarrow (U \wedge U)/H$ as the natural projection, we would find that
$(\pi \circ \varphi)(x,-)=0$, so that $\pi \circ \varphi$ is non-regular; then,
$\urk \calS_{\pi \circ \varphi} \leq n-2$, which shows that $\calS_\varphi$ is non-primitive.

\item Conversely, if $\calS_\varphi$ is non-primitive, then one finds a linear surjection $\pi : V \twoheadrightarrow V'$
to an $(m-1)$-dimensional vector space $V'$ such that $\psi:=\pi \circ \varphi$
satisfies $\urk \calS_\psi <n-1$.
As $m-1 >1+\dbinom{n-2}{2}$ and $\# \K\geq n-1$, Proposition
\ref{transitivityprop2} shows that $\psi$ is non-regular.
This yields a non-zero vector $x$ such that $x \wedge U \subset \calW_\psi$.
As $\calW_\varphi$ is a hyperplane of $\calW_\psi$, one deduces that $x \wedge U_0 \subset \calW_\psi$
for some hyperplane $U_0$ of $U$. Finally, $x \in U_0$, as
$x \not\in U_0$ would yield $\varphi(x,-)=0$.
\end{itemize}
\end{proof}

\subsection{On the equivalence of spaces of the alternating kind}\label{equivalencetocongruencesection}

We have seen that two operator spaces $\calS_\varphi$ and $\calS_\psi$ of the alternating kind are
equivalent if and only if the bilinear maps $\varphi$ and $\psi$ are equivalent.
Here, we investigate when the equivalence of fully-regular alternating bilinear maps imply their congruence.
Note that it suffices to consider bilinear maps with the same source space and the same target space.

\begin{lemme}\label{CNScongruencevsequivalence}
Let $\varphi : U \times U \rightarrow V$ be a fully-regular alternating bilinear map.
Assume that every automorphism $u$ of $U$ such that $\forall x \in U, \; \varphi(u(x),x)=0$
is a scalar multiple of the identity. Then, every fully-regular alternating bilinear map
that is equivalent to $\varphi$ is actually congruent to $\varphi$.
\end{lemme}

\begin{proof}
Let $\psi : U' \times U' \rightarrow V'$ be a fully-regular alternating bilinear map
which is equivalent to $\varphi$. Without loss of generality, we may assume that $U=U'$ and $V=V'$, and then we have automorphisms
$f$, $g$ and $h$, respectively, of $U$, $U$ and $V$ such that
$$\forall (x,y) \in U^2, \; \psi(x,y)=h\bigl(\varphi(f(x),g(y))\bigr).$$
Without loss of generality (using the congruence defined by the triple $(g^{-1},g^{-1},\id)$),
we may assume that $g=\id_U$. As $\forall x \in U, \; \varphi(f(x),x)=0$, we find a scalar $\alpha$ such that $f=\alpha\,\id_U$, and
we deduce that
$$\forall (x,y)\in U^2, \; \psi(x,y)=(\alpha h)\bigl(\varphi(x,y)\bigr).$$
As $f$ is non-zero, $\alpha$ is non-zero whence $\alpha\,h$ is an automorphism of $V$;
one concludes that $\varphi$ and $\psi$ are congruent.
\end{proof}

We finish with more useful sufficient conditions for equivalence to imply congruence:

\begin{prop}\label{CScongruence}
Let $\varphi : U \times U \rightarrow V$ and $\psi : U \times U \rightarrow V$ be fully-regular
alternating bilinear maps, and set $n:=\dim U$.
Assume that $\dim V>\dbinom{n-1}{2}$, or that $\# \K \geq n$ and $\urk \calS_\varphi=n-1$.
Then, $\varphi$ and $\psi$ are equivalent if and only if they are congruent.
\end{prop}

\begin{proof}
Let $u$ be an automorphism of $U$ such that $(x,y) \mapsto \varphi(u(x),y)$ is alternating.
By Lemma \ref{CNScongruencevsequivalence}, it suffices to prove that $u$ is a scalar multiple of the identity.
Assume that $u\not\in \K \id_U$. Then, the set $E$ of eigenvectors of $u$ is included in a $2$-compound of $U$,
and it is even included in a hyperplane of $U$ if $\# \K=2$.
However, given a non-zero vector of $U$ that is not an eigenvector for $u$, the linearly independent
vectors $x$ and $u(x)$ are
both annihilated by $\varphi(x,-)$,
which yields $\rk \varphi(x,-) \leq n-2$. Thus, every non-zero vector $x \in U$ satisfying $\rk \varphi(x,-)=n-1$
must belong to $E$. If $\dim V>\dbinom{n-1}{2}$, this contradicts Lemma \ref{transitivitylemma}.

Let us now assume that $\# \K \geq n$ and $\urk \calS_\varphi=n-1$.
Lemma \ref{spanofrankr} shows that $u$ must be diagonalisable.

Consider a $2$-dimensional subspace $P$ of $U$, and suppose
that the set of all vectors $x \in P$ for which $\rk \varphi(x,-)=n-1$ is included in a $1$-dimensional linear subspace
 of $P$. Then, Lemma \ref{spanofrankr} applied to the operator space $\{\varphi(x,-)\mid x \in P\}$ yields that
$\rk \varphi(x,-)<n-1$ for all $x \in P$.

Assume that $u$ has more than two eigenvalues, consider an eigenvector $x$ of $U$ and two eigenvectors $y$ and $z$
such that the respective eigenvalues of $x$, $y$ and $z$ are pairwise distinct. Then, the eigenvectors of
$u$ in $\Vect(x,x+y+z)$ are included in $\K x$, and one deduces from the above point that $\rk \varphi(x,-)<n-1$.
Therefore, $\rk \varphi(x,-)<n-1$ for all eigenvectors $x$ of $u$, as well as non-eigenvectors.
This contradicts our assumption that $\urk \calS_\varphi=n-1$.

Thus, $u$ has two eigenvalues exactly. If $\dim U=2$, then, as $\varphi$ is fully-regular, we have $\dim V=1$, and hence
$\rk \varphi(x,-)=1$ for all non-zero vector $x \in U$; then, every non-zero vector of $U$ would be an eigenvector of $u$,
which is false.
Therefore, $\dim U>2$. Denote by $\lambda$ and $\mu$ the two eigenvalues of $u$, with $\dim \Ker(u-\lambda \id_U)>1$.
Let $x$ be a non-zero vector of $\Ker(u-\lambda \id_U)$. Choose $y \in \Ker(u-\lambda \id_U) \setminus \K x$,
and $z \in \Ker(u-\mu \id_U)\setminus \{0\}$. Then, one sees that the eigenvectors of $u$ in $\Vect(x,y+z)$ belong to
$\K x$. It follows from the above considerations that $\rk \varphi(x,-)<n-1$. Therefore,
every vector $x \in U$ for which $\rk \varphi(x,-)=n-1$ must belong to $\Ker(u-\mu \id_U)$, and hence those vectors do not span $U$.
This contradicts Lemma \ref{spanofrankr}. Therefore, $u$ is a scalar multiple of the identity, as claimed.
\end{proof}

Noting that $\calS_\varphi$ and $\calS_\psi$ are congruent if and only if $\calV_\varphi$ and $\calV_\psi$ are congruent,
the following known result ensues:

\begin{prop}
Let $\calV_1$ and $\calV_2$ be incompressible linear subspaces of $\Mata_n(\K)$.
Assume either that $\dim \calV_1>\dbinom{n-1}{2}$, or that
$\trk \calV_1=n-1$ and $\# \K \geq n$.
Then, $\calV_1$ and $\calV_2$ are equivalent if and only if they are congruent.
\end{prop}

\subsection{Triality}\label{trialitysection}

For further classification theorems, it is worthwhile to have necessary and sufficient conditions
for an operator space of the alternating kind to be equivalent to the transpose of another such space.
First of all, let us consider regular alternating bilinear maps
$\varphi : U \times U \rightarrow V$ and $\psi : U' \times U' \rightarrow V'$ such that
$\calS_\varphi$ is equivalent to $\calS_\psi^T$; then, $\varphi$ is equivalent to a bilinear map $U' \times (V')^\star \rightarrow (U')^\star$,
and hence $\dim U=\dim U'$, $\dim U=\dim (V')^\star$ and $\dim V=\dim (U')^\star$; one concludes that
$\dim U=\dim V=\dim U'=\dim V'$.

\begin{prop}\label{exclusivitegeneral}
Let $\varphi : U \times U \rightarrow V$ be a fully-regular essentially surjective bilinear map, with $\dim U=\dim V$.
Then, the following conditions are equivalent:
\begin{enumerate}[(i)]
\item The space $\calS_\varphi^T$ is of the alternating kind;
\item The bilinear map $\varphi$ is equivalent to the bilinear map $U \times U \rightarrow U^\star$ associated with a
fully-regular alternating trilinear form
$B : U\times U \times U \rightarrow \K$;
\item The space $\calS_\varphi$ is represented, in well-chosen bases, by a space of alternating matrices;
\item The space $\calS_\varphi^T$ is equivalent to $\calS_\varphi$.
\end{enumerate}
If those conditions hold, then $\calS_\varphi \sim \widehat{\calS_\varphi^T}$ and
hence $\calS_\varphi$ is represented by $\calV_\varphi$
in well-chosen bases.
\end{prop}

\begin{proof}
Implications (iii) $\Rightarrow$ (iv) and (iv) $\Rightarrow$ (i) are obvious.

Assume that (ii) holds, and denote by $t : U^3 \rightarrow \K$ an alternating trilinear form
such that $\theta_t : (x,y)\in U^2 \mapsto t(x,y,-) \in U^\star$ is equivalent to $\varphi$.
Then, given a basis $\calB$ of $U$, one sees that $\calS_{\theta_t}$ is represented by
alternating matrices in $\calB$ and its dual basis. This proves (iii), as $\calS_\varphi \sim \calS_{\theta_t}$.

Finally, assume that (i) holds. We lose no generality in assuming that $V=U^\star$.
Then, there exists a regular alternating bilinear map $\psi : U \times U \rightarrow U^\star$
such that $\calS_\varphi^T$ is equivalent to $\calS_\psi$.
Using $\varphi$ and $\psi$, we define trilinear forms $\widehat{\varphi} : (x,y,z) \mapsto \varphi(x,y)[z]$ and
$\widehat{\psi} : (x,y,z) \mapsto \psi(x,y)[z]$ on $U^3$.
Thus, we obtain automorphisms $f$, $g$ and $h$ of $U$ such that
$$\forall (x,y,z)\in U^3, \; \widehat{\psi}(x,y,z)=\widehat{\varphi}\bigl(f(x),g(z),h(y)\bigr).$$
Setting $\alpha:=h \circ f^{-1}$, we find
$$\forall (x,y)\in U^2, \; \widehat{\varphi}\bigl(x,y,\alpha(x)\bigr)=\widehat{\psi}\bigl(f^{-1}(x),f^{-1}(x),g^{-1}(y)\bigr)=0.$$
Setting $B : (x,y,z) \mapsto \widehat{\varphi}(x,y,\alpha(z))$, one deduces that
$$\forall (x,y,z)\in U^3, \; B(x,x,y)=0=B(x,y,x).$$
As $B$ is trilinear, this shows successively that $B$ is skew-symmetric and that
$\forall (x,y)\in U^2, \; B(y,x,x)=0$, whence $B$ is alternating.
This proves (ii) since $\alpha$ is an automorphism of $U$. We conclude that conditions (i) to (iv) are pairwise equivalent.
\end{proof}

In particular, an operator space of the alternating kind that is equivalent to the transpose of such a space
must stem from an alternating trilinear form.

Here is the most basic example. Consider a $3$-dimensional vector space $U$ together with a non-zero
alternating trilinear form $\varphi : U^3 \rightarrow \K$.
Obviously $\varphi$ is fully-regular, and hence the operator space associated with the bilinear map $\Phi : (x,y) \mapsto \varphi(x,y,-)$
is represented, in well-chosen bases, by a $3$-dimensional space of alternating $3 \times 3$ matrices, i.e.\ by $\Mata_3(\K)$.
By the above theorem, any subspace of $\Mat_3(\K)$ which is of the alternating kind and is
equivalent to the transpose of a space of the alternating kind must be equivalent to $\Mata_3(\K)$.

\paragraph{}
We finish this section by giving sufficient conditions for the equivalence of fully-regular alternating trilinear forms
to imply their congruence.

\begin{prop}
Let $\varphi : U^3 \rightarrow \K$ and $\psi : U^3 \rightarrow \K$ be fully-regular alternating trilinear forms.
Assume that there exists $x \in U$ such that $\rk \varphi(x,-,-)=\dim U-1$ and that $\# \K \geq \dim U$. Then,
$\varphi$ and $\psi$ are equivalent if and only if they are congruent.
\end{prop}

\begin{proof}
The converse implication is obvious. Assume that $\varphi$ and $\psi$ are equivalent, so that we have automorphisms
$f$, $g$ and $h$ of $U$ such that
$$\forall (x,y,z) \in U^3, \; \psi(x,y,z)=\varphi(f(x),g(y),h(z)).$$
Replacing $\varphi$ with the congruent form $(x,y,z) \mapsto \varphi(f(x),f(y),f(z))$, we see that no generality is lost in assuming that
$f=\id_U$. In that situation, we prove that $g$ and $h$ are scalar multiples of the identity.

Consider the regular alternating bilinear maps
$$\Phi : (x,y) \in U^2 \mapsto \varphi(x,y,-)\in U^\star \quad \text{and} \quad
\Psi : (x,y) \in U^2 \mapsto \psi(x,y,-)\in U^\star,$$
and note that $\urk \calS_\Phi=\dim U-1$.
Then,
$$\forall (x,y)\in U^2, \;  \Psi(x,y)=h^T\bigl(\Phi(x,g(y))\bigr).$$
Using the line of reasoning from the proof of Proposition \ref{CScongruence}, one finds a non-zero scalar $\alpha$
such that $g=\alpha \id_U$. Applying the same arguments to the
alternating bilinear maps $(x,y) \in U^2 \mapsto \varphi(x,-,y)\in U^\star$ and
$(x,y) \in U^2 \mapsto \psi(x,-,y)\in U^\star$ yields a non-zero scalar $\beta$ such that $h=\beta \id_U$.
Thus,
$$\forall (x,y,z)\in U^3, \; \psi(x,y,z)=(\alpha\beta)\,\varphi(x,y,z),$$
with $(\alpha \beta)\in \K \setminus \{0\}$, as claimed.
\end{proof}

\section{Classification theorems for minimal LLD spaces}\label{classtheoSection}

In this section, we use the connection between minimal LLD operator spaces
and semi-primitive matrix spaces to obtain
classification theorems for the former. We will be mostly concerned with
classifying reduced spaces of matrices with the column property as that property behaves well with induction processes:
the results on semi-primitive spaces and on primitive spaces will appear as special cases.

We are interested in two kinds of results, the latter of which will, for the most part, be obtained as
corollaries of the former.
\begin{enumerate}[(a)]
\item For arbitrary values of $r$, classify the reduced spaces $\calS \subset \Mat_{m,n}(\K)$ with the column property and upper-rank $r$,
when $m$ is close to the upper bound $\dbinom{r+1}{2}$. As that upper bound is attained by the operator space associated with
$(x,y) \in (\K^{r+1})^2 \mapsto x \wedge y$, it is a reasonable guess that
such spaces should be very close to spaces of the alternating kind when $m$ is ``close to" $\dbinom{r+1}{2}$.

\item Obtain classification results for all primitive spaces with small upper-rank $r$ and a field with
more than $r$ elements. The case $r=2$ is known. We shall give a complete classification for $r=3$.
\end{enumerate}

\subsection{The reduction lemma}\label{reductionlemmasection}

The following result is the key to the classification theorems that will follow.
It is implicit in Atkinson's work \cite{AtkinsonPrim}. For further developments on the topic, it is necessary to
state it precisely and to give a detailed proof of it.

\begin{prop}[Reduction lemma]\label{reductionlemma}
Let $\calS$ be a reduced subspace of $\Mat_{m,n}(\K)$ with $\urk \calS=n-1$.
Set $r:=n-1$ and assume that $\# \K>r$.
Assume that $\calS$ is $r$-reduced. Every matrix $M$ in $\calS$ splits up as
$$M=\begin{bmatrix}
A(M) & C(M) \\
B(M) & [0]_{(m-r) \times 1}
\end{bmatrix}$$
according to the $(r,r)$-decomposition. Set $p:=\dim C(\calS)$ and assume that $p \geq 2$.
Let
$$\mathbf{M}=\begin{bmatrix}
\mathbf{A} & \mathbf{C} \\
\mathbf{B} & [0]_{(m-r) \times 1}
\end{bmatrix},$$
be a generic matrix of $\calS$.

\begin{enumerate}[(a)]
\item If $\rk \mathbf{B}=r-1$, i.e.\ $\urk B(\calS)=r-1$, then $\rk \begin{bmatrix}
\mathbf{A}\mathbf{C} & \mathbf{C}
\end{bmatrix}=1$.

\item Assume that $\rk \begin{bmatrix}
\mathbf{A}\mathbf{C} & \mathbf{C}
\end{bmatrix}=1$.
Then, there exists a space $\calT \sim \calS$ and an integer
$s \leq m$ such that every matrix of $\calT$ splits up as
$$M=\begin{bmatrix}
D(M) & [?]_{s \times (n-p-1)} \\
[0]_{(m-s) \times (p+1)} &  H(M)
\end{bmatrix},$$
where $D(\calT)$ is a semi-primitive subspace of $\Mat_{s,p+1}(\K)$ of the alternating kind with
$\urk D(\calT)=p$, and $H(\calT) \subset \Mat_{m-s,n-p-1}(\K)$ has the column property
(and therefore has upper-rank less than $n-p-1$ if $s<m$).

\item With the assumptions from (b), if $\calS$ is semi-primitive,
then $p=r$ and $\calS$ is of the alternating kind.
\end{enumerate}
\end{prop}

In the course of the proof, and later in this article, we will need the
following lemma, which generalizes a situation that appears in the proof of Theorem B of \cite{AtkinsonPrim}:

\begin{lemme}\label{genericcolinearity}
Let $R$ be a polynomial ring over $\K$,  and let $p \geq 2$ and $d \geq 1$ be integers. Denote by $\L$ the fraction field of $R$.
Let $\mathbf{X}$ and $\mathbf{Y}$ be two vectors of $R^p$ and assume:
\begin{enumerate}[(i)]
\item That the entries of $\mathbf{X}$ are $1$-homogeneous and the ones of $\mathbf{Y}$ are $d$-homogeneous.
\item That $\mathbf{X}$ is not a multiple of a vector of $\K^p$.
\item That $\mathbf{X}$ and $\mathbf{Y}$ are colinear in the $\L$-vector space $\L^p$.
\end{enumerate}
Then $\mathbf{Y}=\mathbf{p}\,\mathbf{X}$ for some $(d-1)$-homogeneous polynomial $\mathbf{p} \in R$.
\end{lemme}

\begin{proof}
Writing $\mathbf{X}=\begin{bmatrix}
\mathbf{x}_1 & \dots & \mathbf{x}_p
\end{bmatrix}^T$ and
$\mathbf{Y}=\begin{bmatrix}
\mathbf{y}_1 & \cdots & \mathbf{y}_p
\end{bmatrix}^T$, we have $\mathbf{X} \neq 0$ and we
deduce from assumption (iii) that there is a \emph{fraction} $\mathbf{p} \in \L$ such that
$\mathbf{Y}=\mathbf{p}\,\mathbf{X}$. Assumption (ii) shows that there are distinct integers $i$ and $j$ such that
$\mathbf{x}_i$ and $\mathbf{x}_j$ are not linearly dependent over $\K$; as they are $1$-homogeneous polynomials over $\K$,
they are mutually prime. As
$\mathbf{p}=\frac{\mathbf{y}_i}{\mathbf{x}_i}=\frac{\mathbf{y}_j}{\mathbf{x}_j}$, one deduces that $\mathbf{p}$
is a polynomial, and it is obviously $(d-1)$-homogeneous.
\end{proof}

\begin{proof}[Proof of Proposition \ref{reductionlemma}]
Without loss of generality, we may assume that $C(\calS)=\K^p \times \{0\}$.
Assume that $\rk \mathbf{B}=r-1$.
The Flanders-Atkinson lemma yields $\mathbf{B}\mathbf{A}\mathbf{C}=0$ and $\mathbf{B}\mathbf{C}=0$, i.e.\
both vectors $\mathbf{A}\mathbf{C}$ and $\mathbf{C}$ belong to the kernel of $\mathbf{B}$, which has dimension $1$.
This proves point (a) since $\mathbf{C}$ is non-zero.

If we now only assume that $\mathbf{A}\mathbf{C}$ and $\mathbf{C}$
are colinear in the fraction field of the considered polynomial ring over $\K$, then,
as $p \geq 2$ and the entries of $\mathbf{A}\mathbf{C}$ are $2$-homogeneous polynomials, Lemma
\ref{genericcolinearity} yields a $1$-homogeneous polynomial $\mathbf{p}$ such that $\mathbf{A}\mathbf{C}=\mathbf{p}\mathbf{C}$.
Let us write every matrix $M$ of $\calS$ as
$$M=\begin{bmatrix}
A'(M) & [?]_{p \times (r-p)} & C'(M) \\
L(M) & [?]_{(r-p) \times (r-p)} & [0]_{(r-p) \times 1} \\
S(M) & [?]_{(m-r) \times (r-p)} & [0]_{(m-r) \times 1}
\end{bmatrix},$$
where $A'(M)$, $L(M)$, $S(M)$ and $C'(M)$ are, respectively, $p \times p$, $(r-p) \times p$, $(m-r) \times p$ and
$p \times 1$ matrices.
We have just found a linear form $\alpha : \calS \rightarrow \K$ such that, for every $M \in \calS$,
$$A'(M)\,C'(M)=\alpha(M)\cdot C'(M) \quad \text{and} \quad L(M)\,C'(M)=0,$$
while $S(M)C'(M)=0$ by the Flanders-Atkinson lemma.
Now, for $M \in \calS$, we set
$$h(M)=\begin{bmatrix}
A'(M) & C'(M) \\
R(M) & [0]_{(m-p)\times 1}
\end{bmatrix}, \quad \text{where $R(M)=\begin{bmatrix}
L(M) \\
S(M)
\end{bmatrix}$,}$$
so that
\begin{equation}\label{polynomialeigenvector}
A'(M)\,C'(M)=\alpha(M)\cdot C'(M) \quad \text{and} \quad
R(M)\,C'(M)=0.
\end{equation}
Polarizing both quadratic identities in \eqref{polynomialeigenvector} leads to
\begin{equation}\label{polynomialeigenvectorpolarized}
\forall (M,N) \in \calS^2, \; A'(M)\,C'(N)+A'(N)\,C'(M)=\alpha(M)\cdot C'(N)+\alpha(N)\cdot C'(M)
\end{equation}
and
\begin{equation}\label{polynomialeigenvectorpolarized2}
\forall (M,N) \in \calS^2, \; R(M)\,C'(N)+R(N)\,C'(M)=0.
\end{equation}
Let $N \in \calS$ be such that $C'(N)=0$.
Since $C'(\calS)=\K^p$, identities \eqref{polynomialeigenvectorpolarized} and
\eqref{polynomialeigenvectorpolarized2} yield $A'(N)X=\alpha(N)X$ and $R(N)X=0$ for all
$X \in \K^p$. Thus, $A'(N)=\alpha(N)\cdot I_p$ and $R(N)=0$.
As $\calS$ contains the matrix $J_r=\begin{bmatrix}
I_r & [0]_{r \times 1} \\
[0]_{(m-r) \times r} & [0]_{(m-r) \times 1}
\end{bmatrix}$, we deduce the following parametrization:
there are linear maps $\varphi : \K^p \rightarrow \Mat_p(\K)$ and $\psi : \K^p \rightarrow \Mat_{m-p,p}(\K)$
such that $h(\calS)$ is the set of all matrices of the form
$$E(X,a):=\begin{bmatrix}
-a I_p+\varphi(X) & X \\
\psi(X) & [0]_{(m-p) \times 1}
\end{bmatrix} \quad \text{with $a \in \K$ and $X \in \K^p$,}$$
and then we have a linear form $\mu : \K^p \rightarrow \K$ such that $\forall X \in \K^p, \; \varphi(X)\,X=\mu(X)\, X$.
Replacing $\varphi$ with $X \mapsto \varphi(X)-\mu(X)I_p$, we may assume that
$\varphi(X)X=0$ for all $X \in \K^p$. Thus, we have
\begin{equation}\label{kernelidentity}
\forall X \in \K^p, \quad \begin{bmatrix}
\varphi(X) \\
\psi(X)
\end{bmatrix} \times X=0.
\end{equation}

Now, consider the bilinear map
$$\Phi : \Bigl(\begin{bmatrix}
X \\
a
\end{bmatrix},\begin{bmatrix}
Y \\
b
\end{bmatrix}\Bigr)\in (\K^{p+1})^2 \; \longmapsto \; E(X,a)\times \begin{bmatrix}
Y \\
b
\end{bmatrix}\in \K^m.$$
Using \eqref{kernelidentity}, we find that $\Phi$ is alternating; on the other hand, for all
$(X,a)\in \K^{p+1}$, the partial map $\Phi((X,a),-)$ is represented in the canonical bases by
$E(X,a)$, which is the zero matrix only if $(X,a)=0$; thus, $\Phi$ is left-regular. Note that
$h(\calS)$ represents the operator space $\bigl\{\Phi((X,a),-) \mid (X,a)\in \K^{p+1}\bigr\}$ in the canonical bases
of $\K^{p+1}$ and $\K^m$ and that it contains a rank $p$ operator.
Denoting by $V'$ the essential range of $\Phi$ and by $s$ its dimension,
decomposing $\K^m=V'\oplus V''$, replacing the canonical basis of $\K^m$ by a basis that is
adapted to this decomposition and permuting the basis of $\K^n$, we reduce the situation to the one where
every matrix of $\calS$ splits up as
$$M=\begin{bmatrix}
D(M) & [?]_{s \times (n-p-1)} \\
[0]_{(m-s) \times (p+1)} &  H(M)
\end{bmatrix},$$
where $D(\calS)$ represents, in well-chosen bases, the fully-regular alternating bilinear map
$\K^{p+1} \times \K^{p+1} \rightarrow V'$ deduced from $\Phi$,
to the effect that $D(\calS)$ is of the alternating kind with $\urk D(\calS)=p$.
As $\calS$ has the column property, one finds that $H(\calS)$ also does and that its upper-rank is
less than $n-p-1$. Thus, statement (b) is established.

Finally, statement (c) is an obvious consequence of statement (b).
\end{proof}

\subsection{First classification theorem for minimal LLD spaces}\label{classtheo1Section}

We are ready to extend Theorem C of \cite{AtkinsonPrim}: this important classification theorem was
later rediscovered by Eisenbud and Harris for $r \leq 3$ \cite[Corollary 2.5]{EisenbudHarris}
(in both \cite{AtkinsonPrim} and \cite{EisenbudHarris}, the theorem was stated for primitive spaces only,
and, in \cite{EisenbudHarris}, for an algebraically closed field only).
As we shall see in later sections, it is crucial that this first classification theorem be stated for spaces with the column property,
instead of primitive spaces only:

\begin{theo}[First classification theorem]\label{classtheo1}
Let $\calS$ be a reduced subspace of $\Mat_{m,n}(\K)$
with the column property. Set $r:=\urk(\calS)$ and assume that $m>\dbinom{r}{2}+1$ and $\# \K>r \geq 2$.
Then, $r=n-1$ and $\calS$ is of the alternating kind.
\end{theo}

Using the duality argument, we deduce the structure of minimal reduced LLD operator spaces with large essential range:

\begin{cor}[First classification theorem for minimal LLD spaces with large essential range]\label{classcor1}
Let $\calS$ be a minimal reduced LLD subspace of $\calL(U,V)$, where $U$ and $V$ are finite-dimensional vector spaces with
$\dim V>\dbinom{\dim \calS-1}{2}+1$ and $\# \K \geq \dim \calS \geq 3$.
Then, $\calS$ is of the alternating kind.
\end{cor}

The critical case when $\dim V=\dbinom{\dim \calS}{2}$ was rediscovered
by Chebotar and \v Semrl \cite[Theorem 1.2]{ChebotarSemrl} with a stricter assumption on the cardinality of $\K$:
they however missed the connection with Atkinson's theorem and produced a proof that is far longer than Atkinson's.

Now, we obtain Theorem \ref{classtheo1} as a rather straightforward byproduct of Proposition \ref{reductionlemma} and of some
earlier general considerations.

\begin{proof}[Proof of Theorem \ref{classtheo1}]
If $r<n-1$, then Theorem \ref{rangelemma} yields $m \leq 1+\dbinom{r}{2}$, which contradicts our assumptions.
Thus, $r=n-1$.

Now, we may assume that $\calS$ is $r$-reduced, and set $p:=\dim \calS e_n$, where $e_n$ is the last vector of the canonical basis of $\K^n$.
As $\calS$ is both reduced and $r$-reduced, we have $1 \leq p\leq r$.
If $1 \leq p \leq r-1$, then Lemma \ref{finemaxdim} and the subsequent remark show that
$m \leq 1+\dbinom{r}{2}$, another contradiction. Therefore, $p=r \geq 2$.

Splitting every matrix $M$ of $\calS$ along the $(r,r)$-decomposition as
$$M=\begin{bmatrix}
A(M) & C(M) \\
B(M) & [0]_{(m-r) \times 1}
\end{bmatrix},$$
we find that $\urk B(\calS) \leq r-1$ and that $B(\calS)$ has the column property.
Assume that $\urk B(\calS) \leq r-2$. Since $\calS$ is reduced, we lose no generality is assuming that we have an integer $q \in \lcro 0,r\rcro$
such that, for every $M \in \calS$, one has
$$B(M)=\begin{bmatrix}
B'(M) & [0]_{(m-r) \times (r-q)}
\end{bmatrix},$$
where $B'(\calS) \subset \Mat_{m-r,q}(\K)$ is reduced with the column property and $\urk B'(\calS)=\urk B(\calS)$.
Then, Theorem \ref{rangelemma} yields
$m-r \leq \dbinom{r-1}{2}$, and hence $m \leq 1+\dbinom{r}{2}$, contradicting our assumptions.

Therefore, $\urk B(\calS) = r-1$, and hence Proposition \ref{reductionlemma} yields that $\calS$ is of the alternating kind.
\end{proof}

Recently, the special case $m=\dbinom{r+1}{2}$ in Theorem \ref{classtheo1} has been spectacularly applied to
provide a short proof of the generalized Gerstenhaber theorem for fields with large cardinality: see \cite{dSPAtkinsontoGerstenhaber}.
An earlier success of the above classification theorem was the determination of the semi-primitive spaces
with upper-rank $2$. Let $\calV$ be a semi-primitive subspace of $\Mat_{m,n}(\K)$, with $m \geq 3$,
$\urk \calV=2$ and $\# \K>2$. By Theorem \ref{classtheo1},
one must have $m=n=3$ and $\calV$ must be of the alternating kind. Moreover, there is exactly on
such subspace up to equivalence. As we have seen in Section \ref{trialitysection},
the space $\Mata_3(\K)$ qualifies. This yields:

\begin{theo}
Let $\calS$ be a semi-primitive subspace of $\Mat_{m,n}(\K)$, with $\urk \calS=2$ and $\# \K>2$.
Then, either $m=2$ or $\calS \sim \Mata_3(\K)$.
If in addition $\calS$ is primitive, then $\calS \sim \Mata_3(\K)$.
\end{theo}

Using the duality argument, Proposition \ref{effetdunetransposition} and the fact that a space $\calS$ of the alternating kind is always
equivalent to its dual operator space $\widehat{\calS}$, one recovers a theorem of Bre\v sar and \v Semrl \cite{BresarSemrl}
which kicked off the systematic study of LLD operator spaces:

\begin{cor}
Let $\calS \subset \calL(U,V)$ be a $3$-dimensional reduced minimal LLD operator space, with
$\# \K>2$. Then, either $\dim V=2$ or $\dim U=\dim V=3$; in the latter case, $\calS$ is represented by $\Mata_3(\K)$ in some bases
of $U$ and $V$.
\end{cor}

\subsection{The full classification of minimal $4$-dimensional LLD spaces}\label{classn=4section}

Here, we assume that $\# \K > 3$. The classification of all primitive spaces with upper-rank
$3$ was achieved by Atkinson \cite{AtkinsonPrim} for algebraically closed fields of characteristic not $2$,
and later rediscovered by Eisenbud and Harris \cite{EisenbudHarris} with no restriction on the characteristic of the field.
Here, we shall give the complete classification with the sole assumption that $\# \K >3$,
the only difference residing in the structure of the spaces of $4 \times 4$ matrices,
where the quadratic structure of the field comes into play.

Let $m$ and $n$ be two integers with $m \geq 4$ and $n \geq 4$, and $\calS$ be a semi-primitive subspace of
$\Mat_{m,n}(\K)$ with upper-rank $3$.
Assume first that $\calS$ is non-primitive. Remember that no primitive space has upper-rank $1$ and that
the sole primitive space with upper-rank $2$ is $\Mata_3(\K)$, up to equivalence.
Using Proposition \ref{reductiontoprimitive}, we deduce that $n=4$ and $\calS$ is equivalent to a subspace of the space
$$\Biggl\{\begin{bmatrix}
L & L' \\
[0]_{3 \times (p-3)} & A
\end{bmatrix} \mid L \in \Mat_{1,p-3}(\K), \; L' \in \Mat_{1,3}(\K), \; A \in \Mata_3(\K)\Biggr\}.$$
The classification, up to equivalence, of such semi-primitive subspaces seems possible though tedious, and we shall not consider it.

From now on, we assume that $\calS$ is primitive.
If $m \geq 5$, then Theorem \ref{classtheo1} shows that $\calS$ is of the alternating kind, and hence $n=4$ and $m \in \{5,6\}$.
If $n \geq 5$, then, as $\calS^T$ is a semi-primitive space with upper-rank $3$,
we find that $m=4$, $n \in \{5,6\}$ and that $\calS^T$ is of the alternating kind.

Thus, only the case $m=n=4$ is left to consider. The following result of Atkinson - which we quickly reprove for completeness -
answers the problem:

\begin{theo}[Atkinson]
Let $\calS$ be a primitive subspace of $\Mat_4(\K)$ with $\urk \calS=3$ and $\# \K>3$.
Then, one of the spaces $\calS$ or $\calS^T$ is of the alternating kind.
\end{theo}

\begin{proof}
We lose no generality in assuming that $\calS$ is $3$-reduced. Then, we choose a generic matrix
$\begin{bmatrix}
\mathbf{A} & \mathbf{C} \\
\mathbf{B} & 0
\end{bmatrix}$ of $\calS$.
The Flanders-Atkinson lemma yields
$$\begin{bmatrix}
\mathbf{B} \\
\mathbf{B}\mathbf{A}
\end{bmatrix} \times \begin{bmatrix}
\mathbf{C} &  \mathbf{A} \mathbf{C}
\end{bmatrix}=0,$$
whence
$$\rk \begin{bmatrix}
\mathbf{B} \\
\mathbf{B}\mathbf{A}
\end{bmatrix}+\rk \begin{bmatrix}
\mathbf{C} &  \mathbf{A} \mathbf{C}
\end{bmatrix}\leq 3.$$
As $\mathbf{B} \neq 0$ and $\mathbf{C} \neq 0$, we have $\rk \begin{bmatrix}
\mathbf{B} \\
\mathbf{B}\mathbf{A}
\end{bmatrix} \geq 1$ and $\rk
\begin{bmatrix}
\mathbf{C} &  \mathbf{A} \mathbf{C}
\end{bmatrix} \geq 1$.
Therefore, one of the matrices $\begin{bmatrix}
\mathbf{B} \\
\mathbf{B}\mathbf{A}
\end{bmatrix}$ or $\begin{bmatrix}
\mathbf{C} &  \mathbf{A} \mathbf{C}
\end{bmatrix}$ has rank $1$.
Transposing $\calS$ if necessary, we may assume that the second matrix has rank $1$. Finally, as $\calS$ is primitive,
Lemma \ref{1decomplemma} shows that $\dim \calS x \geq 2$ for all non-zero vector $x \in \K^4$, and hence
Proposition \ref{reductionlemma} yields that $\calS$ is of the alternating kind.
\end{proof}

Thus, we conclude:

\begin{theo}[Atkinson's classification of primitive spaces with upper-rank $3$]
Let $\calS$ be a primitive subspace $\calS$ of $\Mat_{m,n}(\K)$ with upper-rank $3$ and $\# \K>3$. Then,
one and only one of the following conditions holds:
\begin{enumerate}
\item[(i)] $n=4$, $m \geq 4$ and $\calS$ is of the alternating kind;
\item[(ii)] $m=4$, $n \geq 4$ and $\calS^T$ is of the alternating kind.
\end{enumerate}
\end{theo}

\begin{proof}
We have proved that at least one condition holds. If both hold,
then $m=n=4$ and Proposition \ref{exclusivitegeneral} shows that
$\calS$ is equivalent to a linear subspace of $\Mata_4(\K)$;
this is absurd as no matrix of $\Mata_4(\K)$ has rank $3$ (the rank of an alternating matrix is always even).
\end{proof}

The exclusivity problem for $m=n=4$ was not explained in \cite{AtkinsonPrim}, whereas, in \cite{EisenbudHarris},
it was proved only for algebraically closed fields by relying upon an explicit parametrization of the primitive spaces of the alternating kind.

\paragraph{}
Now, we are essentially left with the problem of classifying, up to equivalence, the
operator spaces associated with fully-regular alternating bilinear maps
$\varphi : \K^4 \times \K^4 \longrightarrow \K^m$ for all $m \in \{4,5,6\}$.
Using Corollary \ref{semiprimitivesmallfield} and Proposition \ref{smallfieldprimitive}, we find that the space $\calS_\varphi$ is always semi-primitive
with upper-rank $3$, and it is primitive if $m>4$.
Moreover, Propositions \ref{significationcongruence} and \ref{CScongruence}
show that the equivalence class of $\calS_\varphi$ is entirely determined by a congruence class
of $(6-m)$-dimensional subspaces of $\Mata_4(\K)$, or, alternatively, by a congruence class of $m$-dimensional subspaces.
It remains to classify such classes and, in the case $m=4$,
to determine which corresponding matrix spaces are primitive.

We denote by $(E_{i,j})_{1 \leq i,j \leq 4}$ the canonical basis of $\Mat_4(\K)$ and, for all
$(i,j)\in \lcro 1,4\rcro^2$, we set
$$A_{i,j}:=E_{i,j}-E_{j,i.}$$

\begin{enumerate}[(1)]
\item The sole $6$-dimensional subspace of $\Mata_4(\K)$ is itself.
Putting the canonical basis
$(A_{i,j})_{1 \leq i<j \leq 4}$ of $\Mata_4(\K)$ in the lexicographical order and
using Remark \ref{parametrage}, we find that an associated matrix space has the following generic matrix:
$$\begin{bmatrix}
-\mathbf{b} & \mathbf{a} & 0 & 0 \\
-\mathbf{c} & 0 & \mathbf{a} & 0 \\
-\mathbf{d} & 0 & 0 & \mathbf{a} \\
0 & -\mathbf{c} & \mathbf{b} & 0 \\
0 & -\mathbf{d} & 0 & \mathbf{b} \\
0 & 0 & -\mathbf{d} & \mathbf{c}
\end{bmatrix}.$$

\item Congruence classes of $1$-dimensional subspaces of $\Mata_4(\K)$ are classified by the rank of their non-zero elements.
As every non-zero $4 \times 4$ alternating matrix has rank $2$ or $4$, we find exactly two such classes:
\begin{enumerate}[(i)]
\item The congruence class of $\K A_{3,4}$; its orthogonal subspace is spanned by the basis
$(A_{1,2},A_{1,3},A_{1,4},A_{2,3},A_{2,4})$, which yields the matrix space with generic matrix
$$\begin{bmatrix}
-\mathbf{b} & \mathbf{a} & 0 & 0 \\
-\mathbf{c} & 0 & \mathbf{a} & 0 \\
-\mathbf{d} & 0 & 0 & \mathbf{a} \\
0 & -\mathbf{c} & \mathbf{b} & 0 \\
0 & -\mathbf{d} & 0 & \mathbf{b}
\end{bmatrix}.$$
\item The congruence class of $\K(A_{1,2}-A_{3,4})$; its orthogonal subspace is spanned by the basis
$(A_{1,2}+A_{3,4},A_{1,3},A_{1,4},A_{2,3},A_{2,4})$, which yields
the matrix space with generic matrix
$$\begin{bmatrix}
-\mathbf{b} & \mathbf{a} & -\mathbf{d} & \mathbf{c} \\
-\mathbf{c} & 0 & \mathbf{a} & 0 \\
-\mathbf{d} & 0 & 0 & \mathbf{a} \\
0 & -\mathbf{c} & \mathbf{b} & 0 \\
0 & -\mathbf{d} & 0 & \mathbf{b}
\end{bmatrix}.$$
\end{enumerate}
\end{enumerate}
It remains to classify the $2$-dimensional subspaces of $\Mata_4(\K)$ up to congruence.
This is where the nature of $\K$ comes into play. Recall that the pfaffian of a $4 \times 4$ alternating matrix
$M=\begin{bmatrix}
0 & a & b & c \\
-a & 0 & d & e \\
-b & -d & 0 & f \\
-c & -e & -f & 0
\end{bmatrix}$ is defined as
$$\Pf(M):=af-be+cd$$
and satisfies
$$\det M=(\Pf (M))^2.$$
The map $\Pf$ is a $6$-dimensional hyperbolic quadratic form.
Recall that a \textbf{similarity} between two quadratic spaces $(E,q)$ and $(F,q')$ is an isomorphism
from $E$ to $F$ for which there exists a non-zero scalar $\lambda$ -- which we call its
similarity factor -- such that $\forall x \in E,\; q'(u(x))=\lambda\, q(x)$ (in other words,
$u$ is an isometry from $(E,\lambda\,q)$ to $(F,q')$).
We write $q \sim q'$ when such a similarity exists - in which case $q$ and $q'$ are called similar -
and $q \simeq q'$ when $q$ and $q'$ are equivalent (i.e.\ isometric).

We recall the following classical result on the positive isometries for the $4 \times 4$ pfaffian:

\begin{theo}\label{pfaffianrotation}
The group $\Ortho^+(\Pf)$ of all positive isometries of $(\Mata_4(\K),\Pf)$ is the set of all transformations of
the form
$$M \mapsto \lambda\,PMP^T, \quad \text{where $P \in \GL_4(\K)$, $\lambda \in \K^*$ and $\det P=\frac{1}{\lambda^2}\cdot$}$$
\end{theo}

For fields of characteristic not $2$, the corresponding statement for the spinor group
of $\Pf$ is entirely explained in \cite[Chapter XXVIII, Section 3.2]{invitquad}
and the method may be adapted so as to yield
the above result. For the characteristic $2$ case, the necessary adaptations of the proof are discussed
in Exercises 47 and 48 of \cite[Chapter XXXIV]{invitquad}.

We are interested in the following corollary whose knowledge does not seem to be widespread:

\begin{theo}\label{pfaffiansimilarity}
The group $\GO^+(\Pf)$ of all positive similarities of $(\Mata_4(\K),\Pf)$ is the set of all transformations of the form
$$M \mapsto \lambda\,PMP^T, \qquad \text{with $\lambda \in \K^*$ and $P \in \GL_4(\K)$.}$$
\end{theo}

Before we obtain this corollary, we must explain what a positive similarity is.
Let $(E,q)$ be an \emph{even-dimensional} regular quadratic space.
Assume first that $\K$ is quadratically closed. Then, no scalar multiple of the identity
is a negative isometry of $(E,q)$, and on the other hand every similarity $u$ of $(E,q)$ splits up as
$u=\mu\,v$ with $\mu \in \K^*$ and $v \in \Ortho(q)$. In that case, $\K^* \Ortho^+(q)$ is a subgroup
of index $2$ in $\GO(q)$, and we denote it by $\GO^+(q)$.
If we do not assume that $\K$ is quadratically closed, then we may still embed $\K$ into
an algebraically closed field $\overline{\K}$, leading to a natural embedding of the group $\GO(q)$ into
$\GO(q_{\overline{\K}})$: the inverse image of $\GO^+(q_{\overline{\K}})$ under this injection
is denoted by $\GO^+(q)$ and it is obviously a subgroup of index $2$ in $\GO(q)$.
In particular, $\GO^+(q)$ contains all the scalar multiples of the identity and all the positive isometries
(beware that those special elements do not generate $\GO^+(q)$ in general!).
Now, we are ready to prove Theorem \ref{pfaffiansimilarity}.

\begin{proof}[Proof of Theorem \ref{pfaffiansimilarity}]
For $\lambda \in \K^*$ and $P \in \GL_4(\K)$, we set $u_{\lambda,P} : M \in \Mata_4(\K) \mapsto \lambda PMP^T$,
which is obviously an automorphism of $\Mata_4(\K)$.
Let us denote by $\overline{\K}$ an algebraic closure of $\K$.
Let $\lambda \in \K^*$ and $P \in \GL_4(\K)$.
Choosing $\mu \in (\overline{\K})^*$ such that $\mu^2=\lambda^2\det P$,
we can write $\lambda PMP^T=\mu(\lambda \mu^{-1})PMP^T$ for all $M \in \Mata_4(\K)$,
and $(\lambda \mu^{-1})^2\det P=1$; therefore, Theorem \ref{pfaffianrotation} applied in $\Mata_4(\overline{\K})$
yields that $u_{\lambda,P}$ is a positive similarity for $\Pf$, with $\lambda^2 \det P$
as its similarity factor.

Conversely, let $s \in \GO^+(\K)$, with similarity factor $\mu$.
Choose $Q \in \GL_4(\K)$ with $\det Q=\mu^{-1}$. Then, $v:=u_{1,Q} \circ s$ is a positive similarity of $\Pf$
with similarity factor $1$, that is $v \in \Ortho^+(\Pf)$. This yields a pair $(\lambda,P)\in \K^* \times \GL_4(\K)$
such that $v=u_{\lambda,P}$, and hence $s=u_{\lambda,Q^{-1}P}$.
\end{proof}

It follows that two subspaces $V_1$ and $V_2$ of $\Mata_4(\K)$ are congruent if and only they
are conjugate under the action of $\GO^+(\Pf)$.
From there, we use Witt's theorem to obtain:

\begin{cor}\label{Wittsimilarity}
Let $V_1$ and $V_2$ be linear subspaces of $\Mata_4(\K)$.
Assume that $V_1$ is not a lagrangian for $\Pf$.
Then, $V_1$ and $V_2$ are congruent if and only if the quadratic forms $\Pf_{|V_1}$ and $\Pf_{|V_2}$ are similar.
\end{cor}

\begin{proof}
The direct implication is obvious from the above description of positive similarities.
Assume now that there exists a non-zero scalar $\lambda$ such that $\Pf_{|V_2} \simeq \lambda\,\Pf_{|V_1}$,
and choose a bijective isometry $u : (V_1,\lambda\Pf_{|V_1}) \overset{\simeq}{\rightarrow} (V_2,\Pf_{|V_2})$.
In particular, $V_2$ is not a lagrangian for $\Pf$.

Since $\Pf$ is hyperbolic, it is isometric to $\lambda\,\Pf$ and hence
Witt's extension theorem yields that $u$ can be extended to an isometry $v$ from $(\Mata_4(\K),\lambda\,\Pf)$ to
$(\Mata_4(\K),\Pf)$. Thus, $v$ is a similarity of $(\Mata_4(\K),\Pf)$.
If $v$ is a positive similarity, then we deduce from Corollary \ref{pfaffiansimilarity}
that $V_1$ and $V_2$ are congruent. Assume now that $v$ is a negative similarity; then, it suffices to find a negative
isometry of $(\Mata_4(\K),\Pf)$ that stabilizes $P_2$: if $P_2$ contains a non-isotropic vector $a$, then we simply
take the reflection alongside $\K a$; ditto if $P_2^\bot$ contains a non-isotropic vector $a$;
if neither case holds, then $P_2$ must be a lagrangian, which is forbidden.
\end{proof}

\begin{Rem}\label{orbitlagrangian}
With the same line of reasoning, one finds that there are at most two orbits of lagrangians.
\end{Rem}

In particular, two planes $P_1$ and $P_2$ of $\Mata_4(\K)$ are congruent if and only if the quadratic forms
$\Pf_{|P_1}$ and $\Pf_{|P_2}$ are similar.
Note also that any $2$-dimensional quadratic form is equivalent to a sub-form of $\Pf$
(a $p$-dimensional quadratic form is always equivalent to a sub-form
of a $2p$-dimensional hyperbolic quadratic form).
Finally, we have seen in Proposition \ref{CSprimitivite} that the operator space associated with a
plane $P \subset \Mata_4(\K)$ is non-primitive if and only if $P$ is congruent to
the space $K$ of matrices described by the generic matrix
$$\begin{bmatrix}
0 & \mathbf{x} & \mathbf{y} & 0 \\
-\mathbf{x} & 0 & 0 & 0 \\
-\mathbf{y} & 0 & 0 & 0 \\
0 & 0 & 0 & 0
\end{bmatrix}.$$
As $K$ is totally isotropic, one sees that a semi-primitive operator space associated with a plane
$P$ of $\Mata_4(\K)$ is primitive if and only if $\Pf_{|P}$ is non-zero.
Thus, it remains to classify the $2$-dimensional quadratic forms up to similarity and
to exhibit a sub-form of $\Pf$ in each class. For fields of characteristic not $2$, there are two
individual classes of degenerate
forms, together with a family of classes of regular forms
indexed over the quotient group $\K^*/(\K^*)^{[2]}$. In our typology, $\text{D}i$
means that we have a degenerate form of rank $i$,
and R stands for ``regular":
\begin{figure}[h]
\begin{center}
\begin{tabular}{| c || c | c | c |}
\hline
Type & D0 & D1 & $\text{R}(\delta)$ \\
\hline
\hline
Representative & $\langle 0,0\rangle$ & $\langle 1,0\rangle$ &$\langle 1,\delta\rangle$, where $\delta \in \K^*$ \\
& & & is uniquely determined by its class in $\K^*/(\K^*)^{[2]}$ \\
\hline
\end{tabular}
\caption{The classification of $2$-dimensional quadratic forms up to similarity, in characteristic not $2$.}
\end{center}
\end{figure}

This classification is easily obtained by noting that similarity preserves the rank and the discriminant, and that
every non-zero quadratic form is similar to one that represents $1$.

\vskip 3mm
For each class, we give a corresponding plane $P$ of $\Mata_4(\K)$, a basis of its orthogonal subspace $P^\bot$
with respect to the symmetric bilinear form $\langle -\mid -\rangle$ of Section \ref{describealternatingsection}
(not the pfaffian!), and a generic matrix
of an associated semi-primitive matrix space; in each case, we specify whether the matrix space is primitive or not.

\begin{center}
\begin{tabular}{| c | c | c | c | c |}
\hline
Type & Basis of $P$ & Basis of $P^\bot$ & Semi-primitive space & Primitivity \\
\hline
\hline
D0 & $A_{2,4},A_{3,4}$ & $A_{1,2},A_{1,3},$ &  &  No \\
& & $A_{1,4},A_{2,3}$ & $\begin{bmatrix}
-\mathbf{b} & \mathbf{a} & 0 & 0 \\
-\mathbf{c} & 0 & \mathbf{a} & 0 \\
-\mathbf{d} & 0 & 0 & \mathbf{a} \\
0 & -\mathbf{c} & \mathbf{b} & 0 \\
\end{bmatrix}$ & \\
\hline
D1 & $A_{1,2}+A_{3,4}$ & $A_{1,2}-A_{3,4},$ &
 & Yes \\
& $A_{1,4}$ & $A_{1,3},A_{2,3},A_{2,4}$ & $\begin{bmatrix}
-\mathbf{b} & \mathbf{a} & \mathbf{d} & -\mathbf{c} \\
-\mathbf{c} & 0 & \mathbf{a} & 0 \\
0 & -\mathbf{c} & \mathbf{b} & 0 \\
0 & -\mathbf{d} & 0 & \mathbf{b}
\end{bmatrix}$ & \\
\hline
$\text{R}(\delta)$ & $A_{1,2}+A_{3,4}$, &
$A_{1,2}-A_{3,4},A_{1,3},$ &
 & Yes \\
& $A_{2,3}+\delta A_{1,4}$ & $A_{1,4}-\delta A_{2,3},A_{2,4}$ & $\begin{bmatrix}
-\mathbf{b} & \mathbf{a} & \mathbf{d} & -\mathbf{c} \\
-\mathbf{c} & 0 & \mathbf{a} & 0 \\
-\mathbf{d} & \delta\, \mathbf{c} & -\delta\, \mathbf{b} & \mathbf{a} \\
0 & -\mathbf{d} & 0 & \mathbf{b}
\end{bmatrix}$ & \\
\hline
\end{tabular}
\end{center}

Assume now that $\K$ has characteristic $2$.
Then, the similarity classes of $2$-dimensional quadratic forms are a bit more complicated.
First of all, we consider those of regular forms.
The mapping $\calP : x \in \K \mapsto x^2+x$ is an endomorphism of the group $(\K,+)$.
Every regular $2$-dimensional quadratic form $q$ is equivalent to $[a,b]$ for some $(a,b)\in \K^2$, and the class of $ab$ in the quotient group
$\K/\calP(\K)$ depends only on $q$: this class is called the Arf invariant of $q$ and is denoted by $\Delta(q)$.
Two regular $2$-dimensional quadratic forms which take the value $1$ are equivalent if and only if they have the same
Arf invariant. Finally, one checks that $\lambda [a,b] \simeq [\lambda a,\lambda^{-1} b]$ for all $\lambda \in \K^*$
and all $(a,b)\in \K^2$,
whence two similar regular $2$-dimensional quadratic forms have the same Arf invariant.
Thus, to each class $\overline{\delta}$ in $\K/\calP(\K)$ corresponds exactly one similarity class of $2$-dimensional regular quadratic forms
over $\K$: the one of $[1,\delta]$.

Now, we consider the situation of a non-regular $2$-dimensional quadratic form $q$.
Note that $\K^{[2]}$ is then a subfield of $\K$.
The polar form of $q$ must be zero since its rank is even, and hence $q$ is simply a $\K$-linear map from $\K \times \K$ to $\K$, where
the target space is equipped not with its standard structure of $\K$-vector space, but with the
one defined by $\alpha.x:=\alpha^2x$.

The equivalence class of $q$ is then fully determined by the
$\K^{[2]}$-linear subspace $\im q:=q(\K \times \K)$ of $\K$.
Thus, the similarity classes of non-regular $2$-dimensional quadratic form $q$
are classified by the orbits of the $d$-dimensional linear subspaces of the $\K^{[2]}$-vector space $\K$,
for $d \in \{0,1,2\}$, under multiplication by the group $\K^*$.
For each $d\in \{0,1\}$, there is exactly one orbit: $\{0\}$ and $\K^{[2]}$, respectively.
In the case $d=2$, we note that each orbit contains a subspace of the form
$\K^{[2]}\oplus \K^{[2]}a$, with $a \in \K \setminus \K^{[2]}$.
Note that $\K^{[2]}\oplus \K^{[2]}a=\K^{[2]}[a]$ is a subalgebra of the $\K^{[2]}$-algebra $\K$.
Therefore, given two non-square elements $a$ and $b$ of $\K$,
having a non-zero scalar $\lambda \in \K$ for which $\K^{[2]}[a]=\lambda \K^{[2]}[b]$
implies that $\lambda \in \K^{[2]}[a]$, and is therefore equivalent to having $b \in \K^{[2]}[a]$.
Introducing the equivalence relation $\simt$ on $\K \setminus \K^{[2]}$ defined as
$$a \simt b \; \Leftrightarrow \; b \in \K^{[2]}[a] \; \Leftrightarrow \; \exists (x,y)\in \K^*\times \K : \; b=ax^2+y^2,$$
we conclude that the orbits, under multiplication by elements of $\K^*$,
of $2$-dimensional linear subspaces of the $\K^{[2]}$-vector space $\K$ are classified by the
equivalence classes of $\K \setminus \K^{[2]}$ for $\simt$.

Thus, we can give the complete classification of $2$-dimensional quadratic forms, up to similarity:

\begin{figure}[h]
\begin{center}
\begin{tabular}{| c || c | c | c | c |}
\hline
Type & D0 & D1 & $\text{D}2(t)$ & $\text{R}(\delta)$  \\
\hline
\hline
&  &  & $\langle 1,t\rangle$, where $t \in \K \setminus \K^{[2]}$ & $[1,\delta]$, where $\delta \in \K$ \\
Representative & $\langle 0,0\rangle$ & $\langle 1,0\rangle$ & is uniquely determined &  is uniquely determined  \\
&  & &  by its class mod.\ $\simt$ & by its class mod.\ $\calP(\K)$ \\
\hline
\end{tabular}
\caption{The classification of $2$-dimensional quadratic forms up to similarity, in characteristic $2$.}
\end{center}
\end{figure}

Now, in the characteristic $2$ case, we can describe the equivalence classes of semi-primitive subspaces of $\Mat_4(\K)$
of the alternating kind with upper-rank $3$: we have two special classes, together with two families
with respective parameters $\overline{t} \in (\K \setminus \K^{[2]})/\simt$ and $\overline{\delta} \in \K/\calP(\K)$.
For the D0 and D1 types, the description is the same one as in the characteristic not $2$ case.
For the D2$(t)$ type, the description is the same one as for the R$(t)$ type in the characteristic not $2$ case.
It only remains to describe a primitive space associated with the R$(\delta)$ class in the characteristic $2$ case:

\begin{center}
\begin{tabular}{| c | c | c | c |}
\hline
Type & Basis of $P$ & Basis of $P^\bot$ & Primitive space \\
\hline
\hline
 & $A_{1,4}+A_{2,3}+A_{3,4}$, & $A_{1,2}+A_{1,3}$,  &  \\
$\text{R}(\delta)$ & $A_{1,3}+\delta A_{2,4}+A_{1,2}$ & $\delta A_{1,2}+A_{2,4}$ , &

$\begin{bmatrix}
\mathbf{b}+\mathbf{c} & \mathbf{a} & \mathbf{a} & 0 \\
\delta \mathbf{b} & \mathbf{d}+\delta \mathbf{a} & 0 & \mathbf{b} \\
\mathbf{d} & \mathbf{c} & \mathbf{b} & \mathbf{a} \\
\mathbf{d} & 0 & \mathbf{d} &  \mathbf{a}+\mathbf{c}
\end{bmatrix}$ \\
& & $A_{1,4}+A_{2,3}$, &  \\
 & & $A_{1,4}+A_{3,4}$ & \\
\hline
\end{tabular}
\end{center}

For a quadratically closed field, whatever its characteristic,
this yields only three similarity classes of quadratic forms: those of $\langle 0,0\rangle$,
$\langle 1,0\rangle$ and the hyperbolic form $(x,y) \mapsto xy$ on $\K^2$.
Thus, we find only two primitive subspaces of $\Mat_4(\K)$ of the alternating kind, up to equivalence.
For the hyperbolic form, we can give a description which works regardlessly of the characteristic:

\begin{center}
\begin{tabular}{| c | c | c | c |}
\hline
Type & Basis of $P$ & Basis of $P^\bot$ & Primitive space \\
\hline
\hline
 & $A_{1,4},A_{2,3}$ & $A_{1,2}$, $A_{1,3}$,  &  \\
Hyperbolic  & & $A_{2,4}$, $A_{3,4}$ &
$\begin{bmatrix}
-\mathbf{b} & \mathbf{a} & 0 & 0 \\
-\mathbf{c} & 0 & \mathbf{a} & 0 \\
0 & -\mathbf{d} & 0 & \mathbf{b} \\
0 & 0 & -\mathbf{d} & \mathbf{c}
\end{bmatrix}$
\\
\hline
\end{tabular}
\end{center}

Theorem 1.2 of Eisenbud and Harris \cite{EisenbudHarris} then comes out as a very special case of our study.

\paragraph{}
Now, we return to the initial issue of classifying reduced minimal LLD subspaces with
dimension $4$. Let $\calS$ be a reduced minimal LLD subspace of $\calL(U,V)$ with dimension $4$, and assume that
$\dim V \geq 4$ (otherwise, we are in the trivial situation).
There are three cases:
\begin{description}
\item Case 1: $\widehat{\calS}$ is represented by a subspace of
$$\K \vee \Mata_3(\K):=\Biggl\{\begin{bmatrix}
a & L \\
[0]_{3 \times 1} & A
\end{bmatrix} \mid a \in \K, \; L \in \Mat_{1,3}(\K), \; A \in \Mata_3(\K)\Biggr\}.$$
The classification of such spaces is possible but has limited interest.
\item Case 2: $\widehat{\calS}$ is of the alternating kind, and then $\calS \sim \widehat{\calS}$
is also of the alternating kind: those spaces have already been described.
\item Case 3: $\widehat{\calS}$ is the transpose of a primitive space of the alternating kind.
\end{description}
In Case 3, we use Proposition \ref{effetdunetransposition} to obtain the corresponding LLD operator spaces,
represented by matrix subspaces of $\Mata_4(\K)$:
\begin{figure}[h]
\begin{center}
\begin{tabular}{| c | c | c |}
\hline
Type & Characteristic of the field & Generic matrix for a corresponding LLD space \\
\hline
\hline
D1 & indifferent &  $\begin{bmatrix}
0 & \mathbf{a} & \mathbf{b} & 0 \\
-\mathbf{a} &  0 & \mathbf{c} & \mathbf{d} \\
-\mathbf{b} &  -\mathbf{c} & 0 & -\mathbf{a} \\
0 &  -\mathbf{d} & \mathbf{a} & 0
\end{bmatrix}$  \\
\hline
R$(\delta)$ & $\neq 2$ & $\begin{bmatrix}
0 & \mathbf{a} & \mathbf{b} & \mathbf{c} \\
-\mathbf{a} &  0 & -\delta\,\mathbf{c} & \mathbf{d} \\
-\mathbf{b} &  \delta\,\mathbf{c} & 0 & -\mathbf{a} \\
-\mathbf{c} &  -\mathbf{d} & \mathbf{a} & 0
\end{bmatrix}$  \\
\hline
D2$(t)$ & $=2$ & $\begin{bmatrix}
0 & \mathbf{a} & \mathbf{b} & \mathbf{c} \\
\mathbf{a} &  0 & t\,\mathbf{c} & \mathbf{d} \\
\mathbf{b} &  t\,\mathbf{c} & 0 & \mathbf{a} \\
\mathbf{c} &  \mathbf{d} & \mathbf{a} & 0
\end{bmatrix}$ \\
\hline
R$(\delta)$ & $=2$ & $\begin{bmatrix}
0 & \mathbf{a}+\delta \mathbf{b} & \mathbf{a} & \mathbf{c}+\mathbf{d} \\
\mathbf{a}+\delta\mathbf{b} & 0 & \mathbf{c} & \mathbf{b} \\
\mathbf{a} & \mathbf{c} & 0 & \mathbf{d} \\
\mathbf{c}+\mathbf{d} & \mathbf{b} & \mathbf{d} & 0
\end{bmatrix}$ \\
\hline
\end{tabular}
\end{center}
\caption{The matrix spaces associated with $4$-dimensional reduced minimal LLD spaces that are not of the alternating kind.}
\end{figure}

\paragraph{}
We finish this section with a little digression.
To compute the different equivalence classes of primitive spaces with upper-rank $3$,
Atkinson \cite{AtkinsonPrim} relied upon Gauger's classification \cite{Gauger} of essentially surjective alternating bilinear
maps from $\K^4 \times \K^4$ to $\K^m$, for $m \in \{4,5,6\}$, over
an algebraically closed field of characteristic not $2$.

In \cite{Gauger}, Gauger also stated a theorem (Theorem 7.15) in which he classified essentially surjective alternating bilinear
maps from $\K^4 \times \K^4$ to $\K^3$, up to congruence, and for algebraically closed fields of characteristic not $2$.
He claimed that there are exactly $6$ congruence classes of such maps, gave an explicit description of each one,
and then performed lengthy tensor computations to prove the result.
His result is known to be incorrect, as the subspaces $I_4$ and $I_6$ in his theorem are isomorphic (we shall explain why later on).
Our techniques yield a neat solution to this problem, as we know from Theorem \ref{pfaffiansimilarity}
that solving it amounts to classifying the orbits of $3$-dimensional subspaces of $\Mata_4(\K)$ under the action of the positive similarity group
$\GO^+(\Pf)$. Remark \ref{orbitlagrangian} shows that there are at most two orbits of lagrangians under this action:
the two lagrangians defined by the respective generic matrices
$$\mathbf{A}_1:=\begin{bmatrix}
0 & \mathbf{a} & \mathbf{b} & \mathbf{c} \\
-\mathbf{a} & 0 & 0 & 0 \\
-\mathbf{b} & 0 & 0 & 0 \\
-\mathbf{c} & 0 & 0 & 0
\end{bmatrix} \quad \text{and} \quad
\mathbf{A}_2:=\begin{bmatrix}
0 & 0 & 0 & 0 \\
0 & 0 & \mathbf{a} & \mathbf{b} \\
0 & -\mathbf{a} & 0 & \mathbf{c} \\
0 & -\mathbf{b} & -\mathbf{c} & 0
\end{bmatrix}$$
are obviously non-congruent as only the first one is incompressible, whence they define the two orbits of lagrangians\footnote{
It is a general fact that for a hyperbolic quadratic form $q$, there are exactly two orbits of lagrangians under the action of
$\GO^+(q)$, and two lagrangians $L_1$ and $L_2$ belong to the same orbit if and only if
$\dim (L_1 \cap L_2)=\frac{\dim q}{2} \mod 2$.
In the case at hand, this helps one rediscover that two transverse lagrangians must belong to two different orbits.}.
As any $3$-dimensional quadratic form is equivalent to a sub-form of $\Pf$,
Proposition \ref{Wittsimilarity} yields that we are reduced to classifying all the non-zero
$3$-dimensional quadratic forms up to similarity.
For a quadratically closed field of characteristic not $2$, there are exactly three such
quadratic forms, up to equivalence (and hence, up to similarity), namely
$\langle 1,0,0\rangle$, $\langle 1,-1,0\rangle$ and $\langle 1,-1,1\rangle$.
For a quadratically closed field of characteristic $2$, there are also exactly three of them:
$\langle 1,0,0\rangle$, $[0,0] \bot \langle 0\rangle$ and $[0,0] \bot \langle 1\rangle$.
Noting that $[0,0]$ is hyperbolic in the characteristic $2$ case, this leads to the following rectification of Gauger's theorem,
with the characteristic $2$ case taken into account:

\begin{theo}\label{type43classification}
Let $\K$ be a quadratically closed field.
Then, up to congruence, there are exactly five $3$-dimensional subspaces of
$\Mata_4(\K)$. They are associated with the following generic matrices:
$\mathbf{A}_1$, $\mathbf{A}_2$,
$$\mathbf{A}_3:=
\begin{bmatrix}
0 & \mathbf{a} & \mathbf{b} & \mathbf{c} \\
-\mathbf{a} & 0 & \mathbf{c} & 0 \\
-\mathbf{b} & -\mathbf{c} & 0 & 0 \\
-\mathbf{c} & 0 & 0 & 0
\end{bmatrix}, \quad
\mathbf{A}_4:=
\begin{bmatrix}
0 & \mathbf{a} & \mathbf{b} & 0 \\
-\mathbf{a} & 0 & 0 & \mathbf{c} \\
-\mathbf{b} & 0 & 0 & 0 \\
0 & -\mathbf{c} & 0 & 0
\end{bmatrix} \quad \text{and} \quad
\mathbf{A}_5:=
\begin{bmatrix}
0 & \mathbf{a} & \mathbf{b} & 0 \\
-\mathbf{a} & 0 & 0 & \mathbf{c} \\
-\mathbf{b} & 0 & 0 & \mathbf{a} \\
0 & -\mathbf{c} & -\mathbf{a} & 0
\end{bmatrix}.$$
\end{theo}

In the characteristic not $2$ case, we find the following correspondence, up to congruence, between Gauger's $I_i$ spaces
(see \cite[Theorem 7.15]{Gauger}), the aforementioned quadratic forms, and the above generic matrices:

\begin{center}
\begin{tabular}{| c || c | c | c | c | c | c |}
\hline
Type of space & $I_1$ & $I_2$ & $I_3$ & $I_4$ & $I_5$ & $I_6$ \\
\hline
\hline
Quadratic form & $\langle 0,0,0\rangle$ & $\langle 0,0,0\rangle$ & $\langle 1,-1,0\rangle$ & $\langle 1,-1,1\rangle$ & $\langle 1,0,0\rangle$
& $\langle 1,-1,1\rangle$
\\
\hline
Generic matrix & $\mathbf{A}_2$ & $\mathbf{A}_1$ & $\mathbf{A}_4$ & $\mathbf{A}_5$ & $\mathbf{A}_3$ & $\mathbf{A}_5$
\\
\hline
\end{tabular}
\end{center}

In particular, this demonstrates that Gauger's theorem is incorrect.

\vskip 3mm
For general fields, congruence classes heavily depend on the quadratic structure of the underlying field.
In the characteristic not $2$ case, it is necessary to classify the $3$-dimensional regular quadratic forms
with discriminant $1$, up to equivalence. For fields of characteristic $2$, it is necessary to classify
the $3$-dimensional subspaces of the $\K^{[2]}$-vector space $\K$, up to multiplication by a non-zero element of $\K$,
and to classify the forms of type $\langle 1\rangle \bot [ a,b]$, up to equivalence.

\subsection{Second classification theorem for minimal LLD spaces}\label{classtheo2Section}

Here, we obtain a second classification theorem that covers a range of values of $m$ that is roughly twice as large as the one in the first classification
theorem:

\begin{theo}[Second classification theorem]\label{classtheo2}
Let $\calS$ be a reduced subspace of $\Mat_{m,n}(\K)$ with the column property.
Set $r:=\urk(\calS)$.
Assume that $\# \K>r\geq 2$ and that $m> 3+\dbinom{r-1}{2}$.
Then, one of the following situations holds:
\begin{enumerate}[(i)]
\item $\calS$ is of the alternating kind and $r=n-1$.
\item There is a space $\calS'\sim \calS$ in which every matrix has the form
$$M=\begin{bmatrix}
[?]_{1 \times r} & [?]_{1 \times (n-r)} \\
H(M) & [0]_{(m-1) \times (n-r)}
\end{bmatrix}$$
and $H(\calS') \subset \Mat_{m-1,r}(\K)$ is of the alternating kind with $\urk H(\calS')=r-1$.
\item One has $r=n-1$, and there is a space $\calS'\sim \calS$ in which every matrix has the form
$$M=\begin{bmatrix}
H(M) & [?]_{m \times 1}
\end{bmatrix},$$
and $H(\calS')\subset \Mat_{m,r}(\K)$ is of the alternating kind with $\urk H(\calS')=r-1$.
\end{enumerate}
Moreover, if $\calS$ is semi-primitive, then Case (iii) cannot hold.
\end{theo}

Note that nothing here is new when $r \leq 3$, for in that case $3+\dbinom{r-1}{2}\geq 1+\dbinom{r}{2}$,
and hence Theorem \ref{classtheo1} shows that Case (i) holds.

The following corollaries are obvious:

\begin{cor}\label{classtheo2cor1}
Let $\calS$ be a primitive subspace of $\Mat_{m,n}(\K)$.
Set $r:=\urk(\calS)$ and assume that $\# \K>r \geq 2$ and $m> 3+\dbinom{r-1}{2}$.
Then, $r=n-1$ and $\calS$ is of the alternating kind.
\end{cor}

\begin{cor}\label{classcor2}
Let $\calS$ be a reduced minimal LLD subspace of $\calL(U,V)$. Set $n:=\dim \calS$
and assume that $\# \K\geq n \geq 3$ and $\dim V>3+\dbinom{n-2}{2}$. Then:
\begin{itemize}
\item Either $\calS$ is of the alternating kind;
\item Or there is a rank $1$ operator $g \in \calS$ such that,
for the canonical projection $\pi : V \twoheadrightarrow V/\im g$, the operator space
$\{\pi \circ f \mid f \in \calS\}$ is an LLD subspace of $\calL(U, V/\im g)$ of the alternating kind
(with dimension $n-1$).
\end{itemize}
\end{cor}

\begin{proof}[Proof of Theorem \ref{classtheo2}]
If $r \leq 3$, then Theorem \ref{classtheo1} shows that Case (i) holds.
In the rest of the proof, we assume that $r>3$.

Assume first that there exists a vector $x \in \K^n$ such that $\dim \calS x=1$.
By Lemma \ref{1decomplemma}, we lose no generality in assuming that, for some $q \in \lcro r,n-1\rcro$,
every matrix $M$ of $\calS$ splits up as
$$M=\begin{bmatrix}
[?]_{1 \times q} & [?]_{1 \times (n-q)} \\
H(M) & [0]_{(m-1) \times (n-q)}
\end{bmatrix},$$
and $H(\calS)$ is reduced, has the column property, and satisfies $\urk H(\calS)=r-1$.
If $q>r$, then Theorem \ref{rangelemma} yields $m-1 \leq 1+\dbinom{r-1}{2}$, which is forbidden.
Therefore $q=r$. As $m-1>1+\dbinom{r-1}{2}$, the first classification theorem yields
that $H(\calS)$ is of the alternating kind, whence Case (ii) holds.

In the rest of the proof, we assume that no vector $x \in \K^n$ satisfies $\dim \calS x=1$.
Without loss of generality, we may also assume that $\calS$ is $r$-reduced,
so that every matrix of $\calS$ splits up as
$$M=\begin{bmatrix}
A(M) & C(M) \\
B(M) & [0]_{(m-r) \times (n-r)}
\end{bmatrix}$$
according to the $(r,r)$-decomposition.
Then, $\urk B(\calS)+\urk C(\calS) \leq r$, and $\urk B(\calS) \leq r-1$ since
$C(\calS)\neq \{0\}$.

\vskip 2mm
\noindent \textbf{Case 1. $\urk B(\calS)=r-1$.} \\
Then, $\urk C(\calS) = 1$.
As no $1$-dimensional subspace of $\K^r$ may contain all the columns of the matrices in $C(\calS)$,
the classical classification of matrix spaces with upper-rank $1$ shows that all the matrices of
$C(\calS)$ vanish everywhere on some common linear hyperplane of $\K^{n-r}$. This yields $n=r+1$ since $\calS$ is reduced.
Set $p:=\dim \calS e_n$, where $e_n$ is the last vector of the canonical basis of $\K^n$.
Then, $2 \leq p \leq r$.
If $2 \leq p \leq r-2$, then Lemma \ref{finemaxdim} and Remark \ref{convexityremark} show that
$m \leq 3+\dbinom{r-1}{2}$, contradicting our assumptions.
If $p=r$, then Proposition \ref{reductionlemma} applies and shows, since $\calS$ is reduced, that
Case (i) holds. \\
Assume now that $p=r-1$. As $r-1 \geq 2$, Proposition \ref{reductionlemma} yields an integer $s \in \lcro p,m\rcro$
and a space $\calT \sim \calS$ in which every matrix $M$ splits up as
$$M=\begin{bmatrix}
D(M) & [?]_{s \times 1} \\
[0]_{(m-s) \times r} & H(M)
\end{bmatrix},$$
where $D(\calT)$ is of the alternating kind with $\urk D(\calT)=r-1$, and $H(\calT) \subset \Mat_{m-s,1}(\K)$
has upper-rank less than $1$. Thus, $H(\calT)=\{0\}$. Since $\calS$ is reduced, it follows that
$m-s=0$, which shows that Case (iii) holds.

\vskip 2mm
\noindent
\textbf{Case 2. $\urk B(\calS)\leq r-2$.} \\
If $B(\calS)$ were reduced, then Theorem \ref{rangelemma} would yield
$m-r \leq 1+\dbinom{r-2}{2}$, and hence $m \leq 3+\dbinom{r-1}{2}$.
Thus, $B(\calS)$ is not reduced. As in the proof of Lemma \ref{1decomplemma},
we may assume that we have found an integer $s \in \lcro 0,r-1\rcro$ such that, for every $M \in \calS$, one has
$$B(M)=\begin{bmatrix}
R(M) & [0]_{(m-r) \times (r-s)}
\end{bmatrix},$$
where $R(\calS)$ is a reduced subspace of $\Mat_{m-r,s}(\K)$ with the column property.
If $s \leq r-2$, then we deduce from Theorem \ref{rangelemma} that $m-r \leq \dbinom{r-2}{2}$, which again is impossible.
Thus, $s=r-1$. As $m-r>1+\dbinom{r-2}{2}$, we deduce from Theorem \ref{rangelemma} that
$\urk R(\calS)=r-2$.

Now, we may find a generic matrix of $\calS$ of the form
$$\mathbf{M}=\begin{bmatrix}
[?]_{(r-1) \times (r-1)} & \mathbf{C'}_0 & \mathbf{C}' \\
[?]_{1 \times (r-1)} & ? & [?]_{1 \times (n-r)} \\
\mathbf{R} & [0]_{(m-r) \times 1} & [0]_{(m-r) \times (n-r)}
\end{bmatrix},$$
where $\mathbf{R}$ is a semi-generic matrix of $R(\calS)$.
Following the $(r,r)$-decomposition, we may also write
$$\mathbf{M}=\begin{bmatrix}
\mathbf{A} & \mathbf{C} \\
\mathbf{B} & [0]_{(m-r) \times (n-r)}
\end{bmatrix}.$$
By the Flanders-Atkinson lemma, we have
$$\begin{bmatrix}
\mathbf{B} \\
\mathbf{B}\mathbf{A}
\end{bmatrix} \times \begin{bmatrix}
\mathbf{C} & \mathbf{A}\mathbf{C}
\end{bmatrix}=0.$$
However, $\rk \begin{bmatrix}
\mathbf{C} & \mathbf{A}\mathbf{C}
\end{bmatrix} \geq \rk \mathbf{C} \geq 1$, while
$\rk \begin{bmatrix}
\mathbf{B} \\
\mathbf{B}\mathbf{A}
\end{bmatrix} \geq \rk \mathbf{B} = r-2$.
Thus, either $\rk \begin{bmatrix}
\mathbf{C} & \mathbf{A}\mathbf{C}
\end{bmatrix}=1$ or $\rk \begin{bmatrix}
\mathbf{B} \\
\mathbf{B}\mathbf{A}
\end{bmatrix}=r-2$. In the first case, we note that $\rk \mathbf{C}=1$, which yields $n=r+1$,
and then one may follow the line of reasoning from Case 1, using point (b) of Proposition \ref{reductionlemma} this time around.
Then, we see that Case (i) or Case (iii) holds by proving, as in Case 1, that $\dim \calS e_n \in \{r-1,r\}$.

Assume finally that $\rk \begin{bmatrix}
\mathbf{B} \\
\mathbf{B}\mathbf{A}
\end{bmatrix} = r-2$. We shall show that this leads to a contradiction.
Firstly,
$$\begin{bmatrix}
\mathbf{B} \\
\mathbf{B}\mathbf{A}
\end{bmatrix}=\begin{bmatrix}
\mathbf{R} & [0]_{(m-r) \times 1} \\
[?]_{(m-r) \times (r-1)} & \mathbf{R}\mathbf{C'}_0
\end{bmatrix}$$
which, as $\rk \begin{bmatrix}
\mathbf{B} \\
\mathbf{B}\mathbf{A}
\end{bmatrix} =r-2=\rk \mathbf{R}$, yields
$$\mathbf{R}\mathbf{C'}_0=0.$$
On the other hand, $\mathbf{B} \mathbf{C}=0$ leads to $\mathbf{R} \mathbf{C'}=0$.
Denoting by $\mathbf{C'}_1,\dots,\mathbf{C'}_{n-r}$ the columns of $\mathbf{C'}$,
this reads $\mathbf{R} \mathbf{C'}_i=0$ for all $i \in \lcro 1,n-r\rcro$.
Note that $\mathbf{C'}_0$ is non-zero as no vector $x \in \K^n$ satisfies $\dim \calS x=1$.
Moreover, the kernel of $\mathbf{R}$ cannot contain a non-zero vector of $\K^{r-1}$
because $R(\calS)$ is reduced. As $\rk \mathbf{R}=r-2$, the vectors $\mathbf{C'}_0$ and $\mathbf{C'}_1$
are colinear in the fraction field of the polynomial ring used to construct $\mathbf{M}$,
and their entries are $1$-homogeneous polynomials. As $r-1 \geq 2$, Lemma \ref{genericcolinearity}
yields $\mathbf{C'}_1=\lambda\,\mathbf{C'}_0$ for some $\lambda \in \K$.
Thus, $\dim \calS (\lambda\, e_r- e_{r+1}) \leq 1$, where $(e_1,\dots,e_n)$ denotes the canonical basis of $\K^n$.
This contradicts an earlier assumption and finishes the proof.
\end{proof}

\section{On the maximal rank in a minimal LLD operator space}\label{maxranksection}

Here, we tackle the maximal rank problem in minimal LLD spaces.
In other words, given such an LLD space, we want to give a good upper bound on
$\urk \calS$ with respect to the dimension of $\calS$.
By the duality argument, this amounts to giving an upper bound on $\trk \widehat{\calS}$.

We will frequently use the basic remark that if $\calS$ is of the alternating kind and $r=\urk \calS$, then
$\trk \calS=r$ since $\calS \sim \widehat{\calS}$.
The following obvious result will also be used often:

\begin{lemme}
Let $\calS$ be an $(r,s)$-decomposed subspace of $\Mat_{m,n}(\K)$, and denote by $\calT$ the lower space
associated with such a decomposition. Then,
$$\trk \calS \leq \trk \calT+r.$$
\end{lemme}

\subsection{A known upper bound}

Here, we reprove and extend the following theorem of Meshulam and \v Semrl \cite{MeshulamSemrlLAA}:

\begin{theo}[Meshulam, \v Semrl]\label{MeshulamSemrlmaxrank}
Let $\calS$ be a minimal LLD operator space with dimension $n$.
Assume that $\# \K \geq n$. Then, $\rk f \leq 1+\dbinom{n-1}{2}$ for all $f \in \calS$.
\end{theo}

Note that Meshulam and \v Semrl were only able to prove this under the tighter condition $\# \K \geq n+2$
as their proof involved the use of Theorem 3.14 of \cite{dSPLLD1}, which at the time was only known to hold
under that condition.

Using the duality argument, Theorem \ref{MeshulamSemrlmaxrank} is a consequence of the following more general theorem,
which we derive from the first classification theorem:

\begin{prop}\label{maxrankprop1}
Let $\calS$ be a reduced linear subspace of $\Mat_{m,n}(\K)$ with the column property.
Set $r:=\urk(\calS)$ and assume that $\# \K>r$. Then,
$$\trk \calS \leq \dbinom{r}{2}+1.$$
\end{prop}

\begin{proof}
If $r<n-1$, then we simply have $m \leq \dbinom{r}{2}+1$ by Theorem \ref{rangelemma}.
If $r=n-1$ and $m>\dbinom{n-1}{2}+1$, then the first classification theorem yields
$\trk \calS = n-1 \leq \dbinom{r}{2}+1$. In any case, the result ensues.
\end{proof}

In \cite{MeshulamSemrlLAA}, it is shown that the upper bound $\dbinom{r}{2}+1$ is optimal for all $r \leq 3$.

\subsection{An improved upper bound (I)}

Using the second classification theorem, we are now able to improve the preceding result for $r \geq 4$:

\begin{theo}\label{maxrankprop2}
Let $\calS$ be a reduced linear subspace of $\Mat_{m,n}(\K)$.
Set $r:=\urk(\calS)$ and assume that $\# \K>r \geq 2$.
\begin{enumerate}[(a)]
\item If $\calS$ has the column property and $r \geq 4$, then $\trk \calS  \leq \dbinom{r}{2}$.
\item If $\calS$ is semi-primitive, then $\trk \calS \leq 3+\dbinom{r-1}{2}$.
\end{enumerate}
\end{theo}

In terms of LLD spaces, point (b) reads as follows:

\begin{cor}\label{maxrankcor2}
Let $\calS$ be a minimal $c$-LLD operator space with dimension $n$. Assume that $\# \K>n-c \geq 2$.
Then, $\urk \calS \leq 3+\dbinom{n-c-1}{2}$.
\end{cor}

\begin{proof}[Proof of Theorem \ref{maxrankprop2}]
Note that $3+\dbinom{r-1}{2} \geq 1+\dbinom{r}{2}$ if and only if $r \leq 3$, and
$3+\dbinom{r-1}{2} \geq \dbinom{r}{2}$ if and only if $r\leq 4$. \\
Using Proposition \ref{maxrankprop1}, we may then assume that $r \geq 4$ and that $\calS$ has the column property.
If $m\leq 3+\dbinom{r-1}{2}$, then we are done.
Assume now that $m>3+\dbinom{r-1}{2}$, so that Theorem \ref{classtheo2} applies to $\calS$.
\begin{itemize}
\item If Case (i) in Theorem \ref{classtheo2} holds, then we have
$\trk \calS  =r \leq 3+\dbinom{r-1}{2}$.
\item If Case (ii) holds, we have $\trk \calS =r \leq  3+\dbinom{r-1}{2}$.
\item If Case (iii) holds, then $\calS$ is not semi-primitive and $\trk \calS \leq m \leq \dbinom{r}{2}$.
\end{itemize}
In any case, the claimed results are established.
\end{proof}

The upper bound in (a) is optimal: with $r\geq 3$, take indeed $m=\dbinom{r}{2}$, together with a semi-primitive subspace
$\calS$ of $\Mat_{m,r}(\K)$ (with upper-rank $r-1$), and denote by $\calT$ the space of all matrices $\begin{bmatrix}
S & Y
\end{bmatrix}$ with $S \in \calS$ and $Y \in \K^m$.
The space $\calT$ has the column property because $\calS$ does. Moreover, $\urk \calT=r$ and
$\dim \calT e_{r+1}=m=\dbinom{r}{2}$, where $e_{r+1}$ is the last vector of the canonical basis of $\K^{r+1}$.
Thus, $\trk \calT=\dbinom{r}{2}$.

As for the upper bound in (b), it is optimal for $r=3$, but probably not for $r=4$ (we conjecture that the best upper bound is $5$ rather than $6$).

\subsection{An improved upper bound (II)}

We finish with a new improved upper bound for semi-primitive spaces of matrices.

\begin{theo}\label{maxrankprop3}
Let $\calS$ be a semi-primitive subspace of $\Mat_{m,n}(\K)$.
Set $r:=\urk(\calS)$ and assume that $\# \K>r \geq 2$.
Then, $\trk \calS \leq 6+\dbinom{r-2}{2}$.
\end{theo}

\begin{cor}\label{maxrankcor3}
Let $\calS$ be a minimal $c$-LLD operator space with dimension $n$. Assume that $\# \K>n-c \geq 2$.
Then, $\urk \calS \leq 6+\dbinom{n-c-2}{2}$.
\end{cor}

\begin{proof}[Proof of Theorem \ref{maxrankprop3}]
If $r \leq 5$, then the result follows from that of Theorem \ref{maxrankprop2} because $6+\dbinom{r-2}{2}
\geq 3+\dbinom{r-1}{2}$. From now on, we assume that $r \geq 6$.

The structure of the rest of the proof is close to that of Theorem \ref{classtheo2}.
First of all, we assume that some $x \in \K^n$ satisfies $\dim \calS x=1$.
By Lemma \ref{1decomplemma}, we have an integer $q \in \lcro r,n-1\rcro$ for which
every matrix $M$ of $\calS$ splits up as
$$M=\begin{bmatrix}
[?]_{1 \times q} & [?]_{1 \times (n-q)} \\
H(M) & [0]_{(m-1) \times (n-q)}
\end{bmatrix},$$
where $H(\calS)$ is a reduced subspace of $\Mat_{m-1,q}(\K)$ with the column property and
$\urk H(\calS)=r-1$.
If $m-1 \leq 3+\dbinom{r-2}{2}$, then we readily have $\trk \calS\leq m \leq 6+\dbinom{r-2}{2}$.
Assume now that $m-1 > 3+\dbinom{r-2}{2}$. As $r-1 \geq 2$, Theorem \ref{classtheo2}
applies to $H(\calS)$. We examine the three cases separately:
\begin{enumerate}
\item In Case (i), we have $\trk H(\calS) = r-1$, and hence
$\trk \calS \leq 1+r-1=r \leq 6+\dbinom{r-2}{2}$.
\item In Case (ii), we have $\trk H(\calS) \leq 1+(r-2)$ and we conclude as in the first case.
\item Assume finally that Case (iii) holds. Then, $q=r$ and we lose no generality in assuming that, for every $M \in \calS$,
we have $H(M)=\begin{bmatrix}
K(M) & [?]_{(m-1) \times 1}
\end{bmatrix}$, where $\rk K(M) \leq r-2$.
Thus, deleting the $r$-th column from every matrix of $\calS$ yields a matrix with rank less than or equal to $r-1$,
which contradicts the assumption that $\calS$ be semi-primitive.
\end{enumerate}

Now, we assume that no vector $x \in \K^n$ satisfies $\dim \calS x=1$.
Without loss of generality, we assume that $\calS$ is $r$-reduced,
and we split every $M \in \calS$ up as
$$M=\begin{bmatrix}
A(M) & C(M) \\
B(M) & [0]_{(m-r) \times (n-r)}
\end{bmatrix},$$
according to the $(r,r)$-decomposition. Note that
$\urk B(\calS)+\urk C(\calS) \leq r$ and $\urk B(\calS) \leq r-1$.

\vskip 2mm
\noindent \textbf{Case 1: $\urk B(\calS)=r-1$.} \\
Then, $\urk C(\calS) \leq 1$, which yields $n=r+1$ (with the same line of reasoning as in the proof of Theorem \ref{classtheo2}).
Proposition \ref{reductionlemma} applies to $\calS$, with $2 \leq p$ since no vector $x \in \K^n$ satisfies $\dim \calS x=1$.
As $\calS$ is semi-primitive, one deduces that $\calS$ is of the alternating kind, and hence
$\trk \calS = r\leq 6+\dbinom{r-2}{2}$.

\vskip 2mm
\noindent \textbf{Case 2: $\urk B(\calS)\leq r-2$.} \\
If $\urk B(\calS)\leq r-3$, then, applying Proposition \ref{maxrankprop1} to a reduced space associated with
$B(\calS)$, we find $\trk B(\calS) \leq  1+\dbinom{r-3}{2}$, and hence
$\trk \calS \leq r+1+\dbinom{r-3}{2}<6+\dbinom{r-2}{2}\cdot$
Until the end of the proof, we assume that $\urk B(\calS)=r-2$.

\vskip 2mm
\noindent \textbf{Subcase 2.1: $B(\calS)$ is reduced.} \\
If $m-r \leq 3+\dbinom{r-3}{2}$, then we readily have $m \leq 6+\dbinom{r-2}{2}$.
Assume now that $m-r >3+\dbinom{r-3}{2}$.
Then, one may apply Theorem \ref{classtheo2} to $B(\calS)$.
As $\urk B(\calS)=r-2$, Case (ii) must hold, and we deduce that $\trk B(\calS) \leq r-2$.
Thus, $\trk \calS \leq 2r-2 \leq 6+\dbinom{r-2}{2}$ (as $r \geq 4$).

\vskip 2mm
\noindent \textbf{Subcase 2.2: $B(\calS)$ is not reduced.} \\
From there, using the fact that $\calS$ is semi-primitive, the line of reasoning from Case 2 in the proof of Theorem \ref{classtheo2}
applies here and shows that $\calS$ is of the alternating kind. This
yields $\trk \calS = r \leq 6+\dbinom{r-2}{2}$.
\end{proof}

Noting that, over the integers, the minimum of the
function $t  \mapsto \dfrac{(t+1)t}{2}+\dfrac{(r-t+1)(r-t)}{2}$
is $\left\lfloor\dfrac{(r+1)^2}{4}\right\rfloor$, we suggest the following conjecture:

\begin{conj}
Let $\calS$ be a semi-primitive subspace of $\Mat_{m,n}(\K)$ with $\# \K>r:=\urk(\calS)$.
Then,
$$\trk \calS  \leq \left\lfloor \frac{(r+1)^2}{4}\right\rfloor.$$
\end{conj}

\end{document}